\documentclass[twoside,11pt]{article}

\usepackage{amsfonts,amsmath,amssymb,enumerate}
\usepackage{euscript,mathrsfs}
\usepackage[dvips]{graphicx}
\usepackage{psfrag,graphicx}
\usepackage{amsthm}
\usepackage[overload]{empheq}

\usepackage{color,graphics}
\usepackage{graphics}
\usepackage{epsfig}
\usepackage{subfigure}
\usepackage{soul}

\newtheorem{theorem}{Theorem}[section]
\newtheorem{lemma}[theorem]{Lemma}

\newtheorem{remark}[theorem]{Remark}

\numberwithin{equation}{section}

\def\Vo{\vbox{\offinterlineskip\hbox{\kern 3pt$\scriptstyle\circ$}
\kern 1pt\hbox{$V$}}}

\def\Ho{\vbox{\offinterlineskip\hbox{\kern 3pt$\scriptstyle\circ$}
\kern 1pt\hbox{$H$}}}

\def\Wo{\vbox{\offinterlineskip\hbox{\kern 3pt$\scriptstyle\circ$}
\kern 1pt\hbox{$W$}}}

\def\eqldef{\overset{\text{\tiny \rm def}}{=}}
\def\tra{\mathsf{T}}

\newcommand{\abs}[2][{}]{\lvert#2\rvert_{#1}}
\def\dd{\;\!\mathrm{d}}    

\newcommand{\Er}{\mathbb{R}}
\newcommand{\En}{\mathbb{N}}

\newcommand{\ve}{\varepsilon}

\newcommand{\norm}[2][{}]{\lVert#2\rVert_{{#1}}}

\renewcommand{\tilde}{\widetilde}

\newcommand{\dive}[1]{\operatorname{div}#1}

\def\tra{\mathsf T}
\def\eqldef{\overset{\text{\tiny\rm def}}{=}}
\def\L{{\rm{L}}}

\def\W{{\rm{W}}}

\def\C{{\mathrm{C}}}

\def\L{{\mathrm{L}}}
\def\V{{\mathrm{V}}}
\def\X{{\mathrm{X}}}
\def\Y{{\mathrm{Y}}}
\def\W{{\mathrm{W}}}

\newcommand{\Del}{\Delta}

\newcommand{\vrR}{\varrho_{R}}
\newcommand{\vuR}{\vu_R}
\newcommand{\vrd}{\varrho_{\delta}}
\newcommand{\vud}{\vu_{\delta}}

\newcommand{\ove}{\bar{\vu}_\ep}

\newcommand{\bfu}{\mathbf{u}}

\newcommand{\ueh}{\vc{u}_{h,\ep}}
\newcommand{\Ve}{V_{\ep}}

\newcommand{\uh}{\vu_h}

\newcommand{\bFormula}[1]{
\begin{equation} \label{#1}}
\newcommand{\eF}{\end{equation}}

\newcommand{\Ov}[1]{\overline{#1}}

\newcommand{\DC}{\C^\infty_c}

\newcommand{\vr}{\varrho}
\newcommand{\vre}{\vr_\ep}

\newcommand{\vue}{\vu_\ep}

\newcommand{\vu}{\vc{u}}
\newcommand{\vc}[1]{{\bf #1}}

\newcommand{\Qme}{Q_\ep}

\newcommand{\Div}{{\rm div}}
\newcommand{\Grad}{\nabla}

\newcommand{\dx}{\,{\rm d} {x}}

\newcommand{\intO}[1]{\int_{\Omega} #1\,\dx}
\newcommand{\intQ}[1]{\int_{Q} #1\,\dx}

\newcommand{\intQe}[1]{\int_{Q_\ep} #1\,\dx}

\newcommand{\intRN}[1]{\int_{\Er^N} #1\,\dx}
\newcommand{\intRNN}[1]{\int_{\Er^{N}} #1\,\dx}

\newcommand{\ep}{\varepsilon}

\newcommand{\R}{\mathbb{R}}

\begin{document}

\title{A rigorous derivation of the stationary
compressible Reynolds equation via the Navier--Stokes equations}

\author{I.~S.~Ciuperca\thanks{Universit\'e de Lyon, CNRS,
Institut Camille Jordan UMR 5208, 43 boulevard du 11 novembre 1918, 
F--69622 Villeurbanne Cedex, France 
(\tt  ciuperca@math.univ-lyon1.fr)}
\and E.~Feireisl\thanks{Institute of Mathematics of the Academy of Sciences of the Czech Republic,
\v Zitn\' a 25, CZ--115 67 Praha 1, Czech Republic
(\tt  feireisl@math.cas.cz)}
\and M.~Jai\thanks{Universit\'e de Lyon, CNRS, INSA de Lyon
Institut Camille Jordan UMR 5208, 20 Avenue A. Einstein, F--69621 Villeurbanne, France 
(\tt mohammed.jai@insa-lyon.fr, apetrov@math.univ-lyon1.fr)}
\and A.~Petrov$^{\ddagger}$}

\pagestyle{myheadings}
\thispagestyle{plain}
\markboth{I.~S.~Ciuperca, E.~Feireisl, M.~Jai, A.~Petrov}
{\small{A derivation of the stationary
compressible Reynolds equation via the Navier--Stokes equations}}

\maketitle

\begin{abstract}
We provide a rigorous derivation of the compressible Reynolds system as a singular limit
of the compressible (barotropic) Navier--Stokes system on a thin domain. In particular, the 
existence
of solutions to the Navier--Stokes system with non--homogeneous boundary conditions 
is shown that may be
of independent interest. Our approach is based on new {\it a priori} bounds available for the pressure
law of hard sphere type. Finally, uniqueness for the limit problem is established in the 1D case.
\end{abstract}

\hspace*{-0.6cm}{\textbf {Key words.}}
compressible fluids, stationary Navier--Stokes equations,
thin films, Bogovskii's operator, lubrication, Reynolds equation

\hspace*{-0.6cm}{\textbf {AMS Subject Classification.}} 35Q30, 76A20, 76D05, 76D08, 76N10

\section{Introduction}
\label{Intro}

The asymptotic behavior of fluids in thin domains is formally described by a
system of equations proposed by O.~Reynolds \cite{Rey86TLAB}. If the fluid is compressible,
the resulting system is nonlinear involving the density, the pressure and the vertical
derivative of the horizontal component of the velocity as unknowns.
A rigorous derivation of the Reynolds system in the compressible case is
substantially hampered by a lack of analytical results concerning solvability of the
compressible (stationary) Navier--Stokes system, where the tangential component
of the velocity is prescribed on the boundary of the physical domain.

Let \(\Omega\subset\Er^N\), \(N=2, 3\), be a bounded domain of class $\C^{2 + \nu}$.
We denote by
\(\mathbf{u}\),  \(p\) and \(\vr\)
the velocity field, pressure and density, respectively.
We introduce also the symmetric part of the velocity gradient
\(\mathbf{e}(\mathbf{u})\eqldef
\frac12(\nabla\mathbf{u}+(\nabla\mathbf{u})^{\tra})\),
the density of the fluid \(\vr\geq 0\)
and the viscosity coefficients \(\lambda>0\) and \(\mu>0\).
Here \((\cdot)^{\tra}\) denotes the transpose of a tensor.
Then the compressible (barotropic) Navier--Stokes system 
describing the distribution of the density and the velocity field reads:
\begin{subequations}
\label{eq:pb1}
\begin{align}
\label{eq:pb11}
&\dive(\vr \mathbf{u})=0\quad\text{in}\quad\Omega,\\
\label{eq:pb12}
&\dive(\vr\mathbf{u}\otimes\mathbf{u})+
\nabla p(\vr)
=\dive(\mathbb{S}(\nabla\mathbf{u}))\quad\text{in}
\quad\Omega,
\end{align}
\end{subequations}
where \(\mathbb{S}(\nabla\mathbf{u})\eqldef 2\mu\mathbf{e}
(\mathbf{u})+\lambda\dive(\mathbf{u})\mathbf{I}\)
is the viscous part of the stress tensor with \(\mathbf{I}\) denoting the identity matrix.
The system \eqref{eq:pb1} is endowed with the
following boundary conditions
\begin{equation}
\label{eq:pb14}
\mathbf{u}=\bar{\mathbf{u}}\quad\text{on}\quad\partial\Omega,
\end{equation}
where \(\partial\Omega\) denotes the boundary of \(\Omega\) and
\(\bar{\mathbf{u}}: \Omega\rightarrow \Er^N\) such that
\begin{equation}
\label{eq:pb15}
\dive(\bar{\mathbf{u}})=0\quad\text{and}\quad
\bar{\mathbf{u}}\cdot \mathbf{n}=0\quad\text{on}\quad\partial\Omega.
\end{equation}
Furthermore, in accordance with the latter condition in \eqref{eq:pb15},
we assume that the total mass of the fluid is given, namely
\begin{equation}
\label{totalmass}
\int_{\Omega}\vr \dd x=M>0.
\end{equation}
Solvability of problem \eqref{eq:pb1}--\eqref{totalmass}
is largely open, in particular when $\bar{\mathbf{u}}$
is large compared to the inverse value of the Reynolds and Mach numbers.
Strong (classical) solutions have been constructed by several authors on
condition of smallness of the data, see e.g. Plotnikov, Ruban, and
Sokolowski \cite{PRS08ICNS,PRS09IBTE} or Piasecki and Pokorn\' y \cite{PiP14SSNF} for
slightly different boundary conditions. If the boundary data are large,
the theory of weak solutions must be used. Although the issue of \emph{compactness}
of solutions to problem  \eqref{eq:pb1}
is nowadays relatively well understood,
see the seminal monograph by Lions \cite{LI4} as well as other numerous
recent extensions of the theory listed in Plotnikov and Weigant \cite{PloWei1}, 
the problem of suitable \emph{a priori} bounds in the case on the non--homogeneous boundary
conditions and a proper \emph{construction} of solutions seems largely open.

Chupin and Sard \cite{ChS12CFEJ} applied the framework proposed by Bresch and
Desjardins \cite{DeB12EGVC},
where the viscosity coefficients depend on the density in a specific way.
This approach requires additional friction term in the momentum equation,
\begin{equation}
\label{i6}
\Div(\vr \vu \otimes \vu) + \Grad p(\vr) = \Div \mathbb{S}(\Grad \vu) - r\vr |\vu| \vu,
\end{equation}
as well an extra boundary condition for the density,
\begin{equation}
\label{i7}
\vr = \vr_b
\end{equation}
where \(\vr_b\)
is constant on each connected component of  \(\partial \Omega\).
Although this ansatz provides (formally) very strong estimates, notably on
$\Grad \vr$, the resulting problem is obviously overdetermined and definitely
{\it not} solvable for any boundary data \eqref{eq:pb14}, \eqref{eq:pb15}
and \eqref{i7}. Indeed we may consider
the fluid domain $\Omega$ to be the space between two concentric balls,
\begin{equation*}
\Omega = \bigl{\{} x \in \R^N:\; 0 < R_1 < |x| < R_2 \bigr{\}},
\end{equation*}
with the boundary conditions
\begin{equation}
\label{i8}
\bar{\vu} = 0\quad\text{and}\quad\vr_b =
\begin{cases}
\vr_1 \ \mbox{if}\ |x| = R_1,\\
\vr_2 \ \mbox{if}\ |x| = R_2
\ \mbox{with two \it{different} constants}\ \vr_1 \ne \vr_2.
\end{cases}
\end{equation}
Here \(R_1>0\) and \(R_2>0\) are two positive constants.
Taking the scalar product of \eqref{i6} with $\vu$, integrating the resulting equation
by parts, and using \eqref{eq:pb11}, we check easily that
$\vu \equiv 0$ in $\Omega$. This in turn yields $\Grad p(\vr) = 0$ in $\Omega$,
meaning $\vr$ constant in $\Omega$ in contrast with \eqref{i8}.

The lack of rigorous results in the framework of weak solutions of \eqref{eq:pb1}--\eqref{totalmass}
is mainly caused by very poor {\it a priori} estimates
available for the density. On the other hand, the density of {\it real} fluids admits natural bounds,
\begin{equation}
\label{i9}
0 \leq \vr \leq \bar{\vr},
\end{equation}
where the lower bound is obvious while the upper bound is imposed by the molecular theory -
the volume of a real fluid cannot be made arbitrarily small,
see e.g. \cite{KVB62CPES}. This restriction is reflected by a general {\it equation of state} of a real
fluid in the form
\begin{equation}
\label{i10}
p(\vr) = \vr \Theta Z(\vr),
\end{equation}
where $\Theta > 0$ is the absolute temperature and $Z$ the compressibility factor,
\begin{equation} \label{i11}
Z(\vr) \to +\infty \quad\text{for}\quad\vr \to \bar{\vr} ,
\end{equation}
see e.g. \cite{KLM04AEHS}. The analysis in the present paper is based on the hypotheses
\eqref{i10} and \eqref{i11}.

The paper is organized as follows.
In Section \ref{math_form}, the mathematical formulation of the problem in terms of
the velocity field and the density is given. Then a family of approximate problems is introduced
and a fixed point procedure applied to show an existence result for these problems.
Next, under appropriate assumptions on the data, we show that the problem
possesses a weak solution which is obtained as a limit of a sequence of solutions of
the approximate problems. In Section \ref{J}, further
properties of the solution are exhibited by using more specific techniques like
the Bogovskii's and extension operators, and an anisotropic version of 
the Sobolev interpolation inequality
(see Section \ref{a}). Then the justification of the Reynods system is derived.
Finally, the uniqueness result for one--dimensional problem follows
from the strong monotonicity of the pressure.

\section{Existence of weak solutions}
\label{math_form}

The weak formulation associated to \eqref{eq:pb1}--\eqref{eq:pb14} is obtained by multiplying
\eqref{eq:pb11} by \(\varphi\in \C^1(\bar{\Omega};\Er)\) and \eqref{eq:pb12} by
\({\mathbf{\psi}}\in \W^{1,2}_0(\Omega;\Er^N)\), respectively. Integrating
formally these identities over \(\Omega\), we end up with the following problem:
\begin{subequations}
\label{weak_form}
\begin{align}
\notag
&
\text{Find }\mathbf{u}\in\W^{1,2}(\Omega;\Er^N)\text{ and }\vr\in\L^{\infty}(\Omega;\Er^+)
\text{, with }\vr\in[0,\bar{\vr}), \text{satisfying}\\\notag&
\mathbf{u}=\bar{\mathbf{u}}\text{ on }\partial\Omega, \
\int_{\Omega}\vr \,\dd x=M \text{, and}\  p(\vr)\in\L^2(\Omega;\Er) \ \text{such that}  \\&
\label{weak_form1}
\forall\varphi\in\C^1(\bar{\Omega};\Er):
\int_{\Omega}\vr \mathbf{u}\cdot \nabla\varphi \,\dd x=0,\\&
\label{weak_form2}
\forall\psi\in \W^{1,2}_0(\Omega;\Er^N){:}
\int_{\Omega}\bigl((\vr\mathbf{u}{\otimes}\mathbf{u}){:}\nabla\psi
+p(\vr)\dive\psi\bigr)\,\dd x
{=}
\int_{\Omega}\mathbb{S}(\nabla\mathbf{u}){:}\nabla\psi\,\dd x.
\end{align}
\end{subequations}
When there is no confusion, we will use simply the notation \(\X(\Omega)\)
instead of  \(\X(\Omega;\Y)\) where \(\X\) is a functional space and \(\Y\) a
vectorial space. Furthermore, we denote 
\begin{equation*}
\vr_M\eqldef \frac{M}{|\Omega|}. 
\end{equation*}
As $\vr < \overline{\vr}$
a compatibility condition \(\vr_M<\bar{\vr}\) must be imposed.
The functional spaces used throughout the present
paper are commonly used in the literature; the reader may consult e.g.
\cite{Adams75SS,Brez83AFTA}.

We assume that the pressure \(p(\zeta):[0,\bar\vr)\rightarrow \Er\)
is continuously on \([0,\bar\vr)\), differentiable on \((0,\bar\vr)\) and satisfies
assumptions:
\begin{subequations}
\label{pression}
\begin{align}
\label{pression1}
&p(0)\geq 0 \quad\text{and}\quad p'(\zeta)\geq 0
\quad\text{for all}\quad\zeta\in(0,\bar{\vr}),\\
\label{pression2}
&\lim_{\zeta\rightarrow \bar{\vr}}(p(\zeta))=+\infty.
\end{align}
\end{subequations}

As a matter of fact, the construction we propose below requires the estimates available
for elliptic operators, notably Laplacian, on bounded domain.
Accordingly, we impose rather strong regularity properties on
$\partial \Omega$.
Our goal is to show the following existence result:
\begin{theorem}
\label{ET1}
Let $\Omega \subset \Er^N$, $N=2,3$ of class $C^{2 + \nu}$ be a bounded domain.
Let the pressure be given as $p = p(\vr)$ satisfying \eqref{pression}.
Let $M > 0$ be given such that
\begin{equation*}
\vr_M < \Ov{\vr}.
\end{equation*}
Let the boundary datum $\bar{\mathbf{u}}$ be given by a function
\begin{equation*}
\bar{\mathbf{u}} \in \W^{1,q}(\Omega)\quad\text{such that}\quad
q > N\quad\text{and}\quad\bar{\mathbf{u}}\cdot\mathbf{n} = 0
\quad\text{on}\quad\partial \Omega.
\end{equation*}
Then problem \eqref{eq:pb1}--\eqref{totalmass} admits a weak solution in the class:
\begin{equation*}
\vr \in \L^\infty(\Omega)\quad\text{with}\quad
\vr\in[0,\bar{\vr})\quad \text{ a.e. in }\quad
\Omega,\quad\mathbf{u}\in \W^{1,2}(\Omega),\quad
p(\vr) \in \L^2(\Omega).
\end{equation*}
\end{theorem}

\subsection{Approximate problems}
\label{E}

We consider for any \(d\in\En^*\) large enough the space
\(X^d\eqldef\text{span}\{\mathbf{w}^1,\ldots, \mathbf{w}^d\}\subset\W_0^{1,2}(\Omega)\)
with \(\{\mathbf{w}^i\}_{i=1}^{\infty}\) a complete orthogonal system in \(\L^2
(\Omega)\) 
such that \(\mathbf{w}^i\) belongs to \(\C^{2,\eta}(\Omega)\) for all \(i\in\En\).
For any \(R>\frac1{\bar\vr}\) 
large enough, let \(T:\Er\rightarrow \Er\) and \(p_R: [0,+\infty[\rightarrow \Er\)
be the cut--off function and a truncation of the pressure function,
respectively, defined as follows:
\begin{equation*}
T(\vr)\eqldef
\begin{cases}
0\quad\text{if}\quad\vr\leq 0,\\
\vr\quad\text{if}\quad\vr\in[0,\bar\vr],\\
\bar{\vr}\quad\text{if}\quad\vr\geq \bar\vr,
\end{cases}
\end{equation*}
and
\begin{equation*}
p_R(\vr)\eqldef
\begin{cases}
p(\vr)&\text{if}\quad\vr\in[0,\bar\vr-\frac1R],\\
p'(\bar{\vr}-\frac1R)(\vr-(\bar\vr-\frac1R))+p(\bar\vr-\frac1R)
&\text{if}\quad\vr>\bar{\vr}-\frac1R.
\end{cases}
\end{equation*}
Let us introduce a small parameter \(\delta>0\). Then we consider the following family of
approximate problems
\begin{subequations}
\label{weak_penalty}
\begin{align}
\notag
&\text{Find } \vc{u} = \mathbf{v} + \bar{\vu},\ \vc{v} \in X^d\text{ and }
\vr\in\W^{2,2}(\Omega)\text{ such that }\\&
\label{weak_penalty1}
\int_{\Omega}\mathbb{S}(\nabla\mathbf{u}){\cdot}\nabla\mathbf{w}^i\,\dd x=
\int_{\Omega}
T(\vr){\mathbf{u}}{\otimes}{\mathbf{u}}){:}\nabla\mathbf{w}^i\,\dd x
+\int_{\Omega}(p_R(\vr){+}\sqrt{\delta}\vr)\dive\mathbf{w}^i\,\dd x\\\notag&
-\delta\int_{\Omega}\nabla(\vr{\mathbf{u}})\cdot\nabla\mathbf{w}^i\,\dd x
-\int_{\Omega}\delta\vr{\mathbf{u}}\mathbf{w}^i\,\dd x
\quad\text{for all}\quad i=1,\ldots, d,\\&
\label{weak_penalty2}
\delta\vr-\delta\Delta
\vr+\dive(T(\vr)\mathbf{u})=\delta\vr_M
\quad\text{in}\quad\Omega,\\&
\label{weak_bound_cond}
\nabla\vr\cdot\mathbf{n}=0\quad\text{on}\quad\partial\Omega,
\end{align}
\end{subequations}
Obviously,
\begin{equation}\label{totalmassweak}
\vc{u} = \Ov{\vc{u}} \ \mbox{on}\ \partial \Omega,\ \mbox{and}\ \int_{\Omega} \vr \ \dx = M.
\end{equation}

Notice that solutions to
problem \eqref{weak_penalty} depend on \(d\), \(R\) and \(\delta\) but, for simplicity,
we denote these solutions by $\vc{u}$ and \(\vr\).

In the sequel, we denote
by the symbol $\rightharpoonup$ the weak convergence and by $\rightarrow$
the strong convergence and by \(\hookrightarrow\)
and  \(\hookrightarrow \hookrightarrow\) the continuous and compact embeddings, respectively.

\subsection{Preliminary results}
\label{Pre_results}

This section is devoted to some preliminary results that will play a crucial role
in the present work.
First, we shall need the following result:
\begin{lemma}
\label{lemma1}
Assume that \(\mathbf{u}\in\W^{1,2}(\Omega)\) such that
\(\mathbf{u}\cdot \mathbf{n}=0\) on
\(\partial\Omega\) holds.
Then there exists a unique solution \(\vr\in\W^{2,2}(\Omega)\) to problem
\eqref{weak_penalty2}--\eqref{totalmassweak}
such that \(\vr\geq 0\). If, in addition, \(\dive\mathbf{u}\in\L^{\infty}(\Omega)\), then there
exists \(\underline{\vr}>0\) such that \(\vr\geq\underline{\vr}\).
Furthermore, the following stability result holds; if
\begin{equation*}
\mathbf{u}_n\rightarrow \mathbf{u}\quad\text{in}\quad\W^{1,2}(\Omega),
\end{equation*}
and \(\vr_n\) is the solution to \eqref{weak_penalty2} corresponding to \(\mathbf{u}_n\),
then we have
\begin{equation*}
\vr_n\rightharpoonup
\vr\quad\text{in}\quad\W^{2,2}(\Omega)\quad\text{weak},
\end{equation*}
where \(\vr\) is the solution corresponding to \(\mathbf{u}\).
\end{lemma}

\begin{proof}
The proof of existence and uniqueness results can be obtained by using an approach
analogous to \cite{NoS04MTCF}.
However for reader's convenience, a detailed proof is given in a more general
setting in the Appendix. Now taking \(\underline{\vr}>0\) such that
\(\underline{\vr}(\delta-\dive\vu)\leq \delta\vr_{M}\), we get by the comparison principle
that \(\vr\geq\underline{\vr}\).
On the other hand, assume that \(\mathbf{u}_n\)
converges strongly to \(\vu\) in \(\W^{1,2}(\Omega)\). It follows from
\eqref{totalmassweak} that
\begin{equation*}
\norm[\L^1(\Omega)]{\vr_n} = M.
\end{equation*}
Furthermore, by means of the standard elliptic estimates, there exists a constant \(C_1 >0\), 
$C_1$ depending on $\delta$, such that
\begin{equation}
\label{eq:exist10}
\norm[\W^{2,2}(\Omega)]{\vr_n}\leq C_1
\bigl(
1+\norm[\L^2(\Omega)]{\dive(T(\vr_n)\mathbf{u}_n)}
\bigr).
\end{equation}
Using the following identity
\begin{equation*}
\dive(T(\vr_n)\mathbf{u}_n)=
T(\vr_n)\dive\mathbf{u}_n+T'(\vr_n)\nabla\vr_n \cdot \mathbf{u}_n,
\end{equation*}
as well as \eqref{eq:exist10}, we get
\begin{equation}
\label{eq:exists11}
\norm[\W^{2,2}(\Omega)]{\vr_n}\leq C_1
\bigl(
1+\bar\vr\norm[\W^{1,2}(\Omega)]{\mathbf{u}_n}+
\norm[\W^{1,4}(\Omega)]{\vr_n}\norm[\W^{1,2}(\Omega)]{\mathbf{u}_n}
\bigr).
\end{equation}
According to the interpolation theory, we infer that for \(\alpha\in]0,1[\), there exists a
constant \(C_{\alpha}>0\) depending on \(\alpha\) such that
\begin{equation}
\label{eq:exists12}
\norm[\W^{1,4}(\Omega)]{\vr_n}\leq C_{\alpha}
\norm[\L^1(\Omega)]{\vr_n}^{\alpha}\norm[\W^{2,2}(\Omega)]{\vr_n}^{1-\alpha}.
\end{equation}
Clearly, introducing \eqref{eq:exists12} into \eqref{eq:exists11} we get that \(\vr_n\)
is bounded in \(\W^{2,2}(\Omega)\), which proves the Lemma.\\
\end{proof}
Let us recall the following version of Sch{\ae}ffer's fixed point:
\begin{lemma}
\label{fix_point}
Let \(X\) be a Banach space and
\(\mathcal{A}: X\rightarrow X\) be a continuous and compact mapping. Assume that for any
\(\theta\in[0,1]\) any fixed point \(\mathbf{u}_{\theta}\in X\) of
\(\theta\mathcal{A}(\mathbf{u}_{\theta})=\mathbf{u}_{\theta}\) is bounded in \(X\)
uniformly with respect to \(\theta\).
Then \(\mathcal{A}\) possesses at least one fixed point in \(X\).
\end{lemma}
\noindent For a detailed proof of the above lemma,
the reader is referred to \cite[Theorem 9.2.4]{Evans90PDE} .

Notice that Poincar\'e and Korn inequalities, enable us
to deduce that there exists a constant \(C_P>0\) such that
\begin{equation}
\label{E12}
\forall\mathbf{v}\in\W^{1,2}_0(\Omega):\,
\|\mathbf{v}\|^2_{\W^{1,2}(\Omega)} \leq C_P
\int_{\Omega}
\mathbb{S}(\Grad \mathbf{v}): \Grad\mathbf{v}\,\dd x.
\end{equation}

\subsection{Existence result associated to approximate problem and uniform bounds}
\label{Exist_unif_bound}

In this section, the existence result to Problem \eqref{weak_penalty}
is proved by using fixed--point precedure and some uniform bounds are
highlighted that will play a crucial role in the next sections.  To this aim,
we define the mapping \(\mathcal{A}\) by
\begin{equation*}
\begin{aligned}
\mathcal{A}: X^d&\rightarrow X^d\\
\tilde{\mathbf{v}}&\mapsto \mathbf{v}=\mathcal{A}(\tilde{\mathbf{v}})
\end{aligned}
\end{equation*}
$\vc{v} = \vu - \Ov{\vu}$, where $\vu$ is a solution of the problem:
\begin{subequations}
\label{eq:exists1}
\begin{align}
\label{eq:exists1_1}
&\int_{\Omega}\mathbb{S}(\nabla\mathbf{u}){\cdot}\nabla\mathbf{w}^i\,\dd x=
\int_{\Omega}
T(\vr)\tilde{\mathbf{u}}{\otimes}\tilde{\mathbf{u}}){:}\nabla\mathbf{w}^i\,\dd x
+\int_{\Omega} (p_R(\vr){+}\sqrt{\delta}\vr)\dive\mathbf{w}^i\,\dd x\\&\notag
-\delta\int_{\Omega}\nabla(\vr\tilde{\mathbf{u}})\cdot\nabla\mathbf{w}^i\,\dd x
-\int_{\Omega}\delta\vr\tilde{\mathbf{u}}\mathbf{w}^i\,\dd x
\quad\text{for all}\quad i=1,\ldots, d,\\&
\label{eq:exists1_2}
\mathbf{u}=\bar{\mathbf{u}}\quad\text{on}\quad\partial\Omega,
\end{align}
\end{subequations}
where we have set
\(\tilde{\mathbf{u}}\eqldef \tilde{\mathbf{v}}+\bar{\mathbf{u}}\).
Here \(\vr\in\W^{2,2}(\Omega)\), $\vr \geq 0$ is the unique solution of the problem
\begin{subequations}
\label{eq:exists2}
\begin{align}
\label{eq:exists2_1}
&\delta\vr-\delta\Delta
\vr+\dive(T(\vr)\tilde{\mathbf{u}})=\delta\vr_M
\quad\text{in}\quad\Omega,\\&
\label{eq:exists2_2}
\nabla\vr\cdot\mathbf{n}=0\quad\text{on}\quad\partial\Omega,\\&
\label{eq:exists2_3}
\int_{\Omega}\vr\dd x=M.
\end{align}
\end{subequations}
Here above the finite dimensional space \(X^d\) may be endowed with the Hilbert topology of 
\(\W^{1,2}_0(\Omega)\).
Thus the system \eqref{eq:exists1_1} has the following form:
\begin{equation*}
\int_{\Omega}\mathbb{S}(\nabla\mathbf{u})\cdot\nabla\mathbf{w}^i\,\dd x=
\langle \mathbf{F}^i,\mathbf{w}^i\rangle_{\W^{-1,2}(\Omega),\W_0^{1,2}(\Omega)}
\quad\text{for all}\quad i=1,\ldots, d,
\end{equation*}
where \(\mathbf{F}^i\) is an element of the dual \(\W^{-1,2}(\Omega)\). 
Consequently, we may deduce from
Lemma \ref{lemma1}, \eqref{E12} and Lax--Milgram theorem, that \(\mathcal{A}\) is well--defined.
Note that the continuity of \(\mathcal{A}\) follows from Lemma \ref{lemma1} while
the compactness of \(\mathcal{A}\) comes from the fact that \(X^d\) is a finite--dimensional space.

Finally, to prove existence of a fixed point of \(\mathcal{A}\), it suffices to verify 
the last assumption of Lemma \ref{fix_point}. 
Suppose that for an arbitrary \(\theta\in]0,1]\), there exists \(\mathbf{v}_{\theta}\)
a fixed point of \(\theta\mathcal{A}\). For simplicity, the subscript \(\theta\) will be omitted in the
sequel. Then we consider the following problem:
\begin{subequations}
\label{eq:exists22}
\begin{align}
\label{eq:exists22_1}
&\delta\vr-\delta\Delta
\vr+\dive(T(\vr)\mathbf{u})=\delta\vr_M
\quad\text{in}\quad\Omega,\\&
\label{eq:exists22_3}
\int_{\Omega}\mathbb{S}(\nabla\mathbf{v}+\theta\nabla\bar{\vu}):\nabla\mathbf{w}^i\,\dd x=
\theta\int_{\Omega}
(T(\vr)\mathbf{u}\otimes\mathbf{u}):\nabla\mathbf{w}^i\,\dd x\\&\notag
+\theta\int_{\Omega}(p_{R}(\vr)+\sqrt{\delta}\vr)\dive\mathbf{w}^i\,\dd x\\&\notag
-\theta\delta\int_{\Omega}\nabla(\vr\mathbf{u})\cdot\nabla\mathbf{w}^i\,\dd x
-\theta\int_{\Omega}\delta\vr\mathbf{u}\mathbf{w}^i\,\dd x
\quad\text{for all}\quad i=1,\ldots, d,\\&
\label{eq:exists22_4}
\nabla\vr\cdot\mathbf{n}=0\quad\text{and}\quad
\mathbf{u}=\bar{\mathbf{u}}\quad\text{on}\quad\partial\Omega,\\&
\label{eq:exists22_5}
\int_{\Omega}\vr\,\dd x=M.
\end{align}
\end{subequations}
Since \(\dive\bar{\mathbf{u}}=0\), we may deduce from \eqref{eq:exists22_3} that
\begin{equation}
\label{eq:exists4}
\begin{aligned}
&\int_{\Omega}(\mathbb{S}(\nabla\mathbf{u})-
\mathbb{S}(\nabla\bar{\mathbf{u}}))
:\nabla(\mathbf{u}-\bar{\mathbf{u}})\,\dd x
+\theta\int_{\Omega}
\dive(T(\vr)\mathbf{u}\otimes\mathbf{u})\cdot(\mathbf{u}-\bar{\mathbf{u}})\,\dd x\\&
=\theta\int_{\Omega} (p_{R}(\vr)+\sqrt{\delta}\vr)\dive\mathbf{u}\,\dd x
-\theta\delta\int_{\Omega}\nabla(\vr\mathbf{u}):
\nabla(\mathbf{u}-\bar{\mathbf{u}})\,\dd x\\&
-\theta\int_{\Omega}\delta\vr\mathbf{u}\cdot(\mathbf{u}-\bar{\mathbf{u}})\,\dd x
-\theta\int_{\Omega}\mathbb{S}(\nabla\bar{\mathbf{u}}):\nabla(\mathbf{u}-\bar{\mathbf{u}})\,\dd x.
\end{aligned}
\end{equation}
Notice that the renormalized form of \eqref{eq:exists22_1} is given by
\begin{equation}
\label{E6a}
\begin{aligned}
&\Div( H(\vr) \vu ) +(G'(\vr) T(\vr) - H(\vr)) \Div\vu\\& -
\delta G'(\vr) \Delta \vr+ \delta G'(\vr) \vr = \delta\vr_M G'(\vr),
\end{aligned}
\end{equation}
where $H$ and $G$ are Lipschitz functions satisfying
\begin{equation*}
H'(\vr) = G'(\vr) T'(\vr).
\end{equation*}
Now we choose for any \(\vr>0\), \(G\) and \(H\) as follows:
\begin{equation*}
G(\vr)=G_{R,\delta}(\vr)
\quad\text{and}\quad
H(\vr)\eqldef -p_R(\vr_M)-\sqrt{\delta}\vr_M+\int_{\vr_M}^{\vr}G_{R,\delta}'(y)T'(y)\,\dd y,
\end{equation*}
where
\begin{equation*}
G_{R,\delta}'(\vr)\eqldef \int_{\vr_M}^{\vr}\frac{p_R'(y)+\sqrt{\delta}}{T(y)}\,\dd y.
\end{equation*}
Then we infer that
\begin{equation}
\label{E8}
G_{R,\delta}'(\vr) T(\vr) - H(\vr) = p_R(\vr) + \sqrt{\delta} \vr.
\end{equation}
We also set
\begin{equation*}
G_{\delta}'(\vr)\eqldef\int_{\vr_M}^{\vr}\frac{p'(y)+\sqrt{\delta}}{T(y)}\,\dd y.
\end{equation*}
Consequently, integrating \eqref{E6a}, we find
\begin{equation}
\label{E9}
\intO{(p_R(\vr) {+} \sqrt{\delta} \vr) \Div\vu}
=\intO{\bigl(\delta G_{R,\delta}'(\vr) \Delta \vr {-} \delta G_{R,\delta}'(\vr) \vr
{+} \delta\vr_M G_{R,\delta}'(\vr)\bigr)}.
\end{equation}
Since \(\mathbf{v}\) is regular enough, we may deduce from Lemma \ref{lemma1}
that \(\vr\geq\underline{\vr}\) with \(\underline{\vr}>0\) and it follows that
\(G_{R,\delta}'(\vr)\in\L^{\infty}(\Omega)\).
We multiply \eqref{E9} by \(\theta\) and we use the resulting expression
in \eqref{eq:exists4} to get
\begin{equation}
\label{eq:exists7}
\begin{aligned}
&\int_{\Omega}(\mathbb{S}(\nabla\mathbf{u})-
\mathbb{S}(\nabla\bar{\mathbf{u}}))
:\nabla(\mathbf{u}-\bar{\mathbf{u}})\,\dd x+
\delta\theta\int_{\Omega}\vr G_{R,\delta}'(\vr)\dd x\\&
+\delta\theta\int_{\Omega}\abs{\nabla\vr}^2G_{R,\delta}''(\vr)\,\dd x
=\theta\delta\vr_M
\int_{\Omega}G_{R,\delta}'(\vr)\,\dd x\\&
-\theta\int_{\Omega}
\dive(T(\vr)\mathbf{u}\otimes\mathbf{u})
\cdot(\mathbf{u}-\bar{\mathbf{u}})\,\dd x
-\delta\theta\int_{\Omega}\nabla(\vr\mathbf{u}):
\nabla(\mathbf{u}-\bar{\mathbf{u}})\,\dd x\\&
-\theta\delta\int_{\Omega}\vr\mathbf{u}\cdot(\mathbf{u}-\bar{\mathbf{u}})\,\dd x
-\theta\int_{\Omega}\mathbb{S}(\nabla\bar{\mathbf{u}}):\nabla(\mathbf{u}-\bar{\mathbf{u}})\,\dd x.
\end{aligned}
\end{equation}
Notice that \(\mathbf{u}\cdot\mathbf{n}=0\) on \(\partial\Omega\) implies that
\begin{equation}
\label{eq:exists8}
-\int_{\Omega}\dive(T(\vr)\mathbf{u}\otimes\mathbf{u})\cdot\mathbf{u}\,\dd x=
-\frac12\int_{\Omega}\dive(T(\vr)\mathbf{u})\abs{\mathbf{u}}^2\,\dd x.
\end{equation}
Carrying \eqref{eq:exists22_1} into \eqref{eq:exists8}, we get
\begin{equation}
\label{eq:exists9}
-\int_{\Omega}\dive(T(\vr)\mathbf{u}\otimes\mathbf{u})\cdot\mathbf{u}\,\dd x
=\frac12\int_{\Omega}\bigl(\delta\vr-\delta\Delta
\vr-\delta\vr_M\bigr)\abs{\mathbf{u}}^2\,\dd x.
\end{equation}
Furthermore, the boundary condition \eqref{eq:exists22_4} leads to
\begin{equation}
\label{eq:exists10}
-\frac12\int_{\Omega}\abs{\mathbf{u}}^2\Delta\vr\,\dd x=
\int_{\Omega}{\mathbf{u}}\cdot(\nabla\mathbf{u}\cdot\nabla\vr)\,\dd x.
\end{equation}
Inserting \eqref{eq:exists10} into \eqref{eq:exists8}, we find
\begin{equation}
\label{eq:exists11_1}
\begin{aligned}
&-\int_{\Omega}\dive(T(\vr)\mathbf{u}\otimes\mathbf{u})
\cdot\mathbf{u}\,\dd x=
\frac{\delta}2\int_{\Omega}\vr\abs{\mathbf{u}}^2\,\dd x\\&
+
\delta\int_{\Omega}\mathbf{u}\cdot(\nabla\mathbf{u}\cdot\nabla\vr)\,\dd x
-\frac{\delta\vr_M}2\int_{\Omega}\abs{\mathbf{u}}^2\,\dd x.
\end{aligned}
\end{equation}
On the other hand, we have
\begin{equation}
\label{eq:exists12_1}
\begin{aligned}
&\int_{\Omega}\dive(T(\vr)\mathbf{u}\otimes\mathbf{u})
\cdot\bar{\mathbf{u}}\,\dd x=
\int_{\Omega}(
-\delta\vr+\delta\Delta\vr+\delta\vr_M)
\mathbf{u}\cdot\bar{\mathbf{u}}\,\dd x\\&+
\int_{\Omega}
T(\vr)(\mathbf{u}\cdot\nabla\mathbf{u})\cdot\bar{\mathbf{u}}\,\dd x=
-\delta\int_{\Omega}\vr\mathbf{u}\cdot\bar{\mathbf{u}}\,\dd x-
\delta\int_{\Omega}(\nabla\vr\cdot\nabla{\mathbf{u}})
\cdot\bar{\mathbf{u}}\,\dd x\\&-
\delta\int_{\Omega}(\nabla\vr\cdot\nabla\bar{\mathbf{u}})\cdot\mathbf{u}\,\dd x
+\delta\vr_M \int_{\Omega}\mathbf{u}\cdot\bar{\mathbf{u}}\,\dd x
+\int_{\Omega}T(\vr)(\mathbf{u}\cdot\nabla\mathbf{u})\cdot\bar{\mathbf{u}}\,\dd x.
\end{aligned}
\end{equation}
We may also observe that
\begin{equation}
\label{eq:exists13_1}
\begin{aligned}
&\int_{\Omega}\nabla(\vr\mathbf{u}):
\nabla(\mathbf{u}-\bar{\mathbf{u}})\,\dd x\\&=
\int_{\Omega}\mathbf{u}\cdot\nabla(\mathbf{u}-\bar{\mathbf{u}})\cdot\nabla\vr\,\dd x
+\int_{\Omega}\vr\nabla\mathbf{u}:\nabla(\mathbf{u}-\bar{\mathbf{u}})\,\dd x.
\end{aligned}
\end{equation}
According to \eqref{eq:exists11_1}, \eqref{eq:exists12_1} and \eqref{eq:exists13_1} in
\eqref{eq:exists7}, we get
\begin{equation*}
\begin{aligned}
&\int_{\Omega}(\mathbb{S}(\nabla\mathbf{u})-
\mathbb{S}(\nabla\bar{\mathbf{u}}))
:\nabla(\mathbf{u}-\bar{\mathbf{u}})\,\dd x+
\delta\theta\int_{\Omega}\vr G_{R,\delta}'(\vr)\,\dd x+\delta\theta
\int_{\Omega}\abs{\nabla \vr}^2G_{R,\delta}''(\vr)\,\dd x\\&
+\frac{\theta\delta\vr_M}2\int_{\Omega}\abs{\mathbf{u}}^2\,\dd x+
\delta\theta\int_{\Omega}\vr\abs{\nabla\mathbf{u}}^2\,\dd x
+\frac{\delta\theta}2\int_{\Omega}
\vr\abs{\mathbf{u}}^2\,\dd x=
\theta\delta\vr_M\int_{\Omega}G_{R,\delta}'(\vr)\,\dd x\\&+
\theta\delta\vr_M\int_{\Omega}\mathbf{u}\cdot\bar{\mathbf{u}}\,\dd x
-\delta\theta\int_{\Omega}(\nabla\vr\cdot\nabla\mathbf{u})
\cdot\bar{\mathbf{u}}\,\dd x
+\theta\int_{\Omega}T(\vr)(\mathbf{u}\cdot\nabla\mathbf{u})\cdot\bar{\mathbf{u}}\,\dd x\\&
+\delta\theta\int_{\Omega}\vr\nabla\mathbf{u}:\nabla\bar{\mathbf{u}}\,\dd x-\theta
\int_{\Omega}\mathbb{S}(\nabla\bar{\mathbf{u}}):\nabla(\mathbf{u}-\bar{\mathbf{u}})\,\dd x.
\end{aligned}
\end{equation*}
By using \eqref{E12}, we obtain
\begin{equation}
\label{eq:exists15}
\begin{aligned}
&\frac1{C_P}\| \vu- \bar{\vu} \|^2_{\W^{1,2}(\Omega)}
+\delta\theta\int_{\Omega}\vr\abs{\nabla \vu}^2\,\dd x
+
\delta\theta\int_{\Omega}\vr G_{R,\delta}'(\vr)\,\dd x\\&
+\delta\theta
\int_{\Omega}\Bigl(\frac{p'(\vr)+\sqrt{\delta}}{T(\vr)}\Bigr)
\abs{\nabla \vr}^2\,\dd x+\frac{\theta\delta\vr_M}2\int_{\Omega}\abs{\mathbf{u}}^2\,\dd x
\leq
\theta\delta\vr_M\int_{\Omega}G_{R,\delta}'(\vr)\,\dd x\\&
-\delta\theta\int_{\Omega}(\nabla\vr\cdot\nabla\mathbf{u})
\cdot\bar{\mathbf{u}}\,\dd x
+\theta\int_{\Omega}T(\vr)(\mathbf{u}\cdot\nabla\mathbf{u})
\cdot\bar{\mathbf{u}}\,\dd x+
\theta\delta\vr_M\int_{\Omega}\vu\cdot\bar{\mathbf{u}}\,\dd x\\&
+\delta\theta\int_{\Omega}\vr\nabla\mathbf{u}:\nabla\bar{\mathbf{u}}\,\dd x-\theta
\int_{\Omega}\mathbb{S}(\nabla\bar{\mathbf{u}}):\nabla(\mathbf{u}-\bar{\mathbf{u}})\,\dd x.
\end{aligned}
\end{equation}
The next observation is that
\begin{equation*}
\begin{aligned}
\vr_M \int_{\Omega}
G_{R,\delta}'(\vr)\,\dd x&\leq \vr_M \int_{\{\vr\geq \vr_M\}}G_{R,\delta}'(\vr)\,\dd x\\
&\leq \vr_M\int_{\{\vr_M\leq\vr\leq\lambda\vr_M\}}
G_{R,\delta}'(\vr)\, \dd x +
\vr_M\int_{\{\vr> \lambda \vr_M\}} G_{R,\delta}'(\vr)\,\dd x\\
&\leq M G_{R,\delta}'(\lambda\vr_M)
+ \frac{1}{\lambda}\int_{\{\vr>\lambda\vr_M\}} \vr G_{R,\delta}'(\vr)\,\dd x\\
&\leq M G_{R,\delta}'(\lambda\vr_M)
+ \frac{1}{\lambda} \int_{\{\vr\geq\vr_M\}}{\vr G_{R,\delta}'(\vr)}\,\dd x,
\end{aligned}
\end{equation*}
where $\lambda > 1$ has been chosen in such a way that
\(\lambda \vr_M < \bar{\vr}\).
Note that $G_{R,\delta}$ is independent of $R$ in the interval
$[0, \lambda\vr_M]$ as soon as $\bar{\vr} - \frac{1}{R} > \lambda
\vr_M$. Consequently, going back to \eqref{eq:exists15}, we obtain
\begin{equation}
\label{eq:exists16}
\begin{aligned}
&\frac1{C_P}\| \vu - \bar{\vu} \|^2_{\W^{1,2}(\Omega)}
+\delta\theta\int_{\Omega}\vr\abs{\nabla\vu}^2\,\dd x
+
\delta\omega\theta\int_{\{\vr\geq\vr_M\}}\vr G_{R,\delta}'(\vr)\,\dd x\\&
+\delta\theta
\int_{\Omega}\Bigl(\frac{p'(\vr)+\sqrt{\delta}}{T(\vr)}\Bigr)
\abs{\nabla \vr}^2\,\dd x+\frac{\theta\delta\vr_M}2\int_{\Omega}\abs{\mathbf{u}}^2\,\dd x\leq
\delta\theta\vr_M\int_{\Omega}\mathbf{u}\cdot\bar{\vu}\,\dd x\\&
-\delta\theta\int_{\Omega}(\nabla\vr{\cdot}\nabla\mathbf{u})
\cdot\bar{\mathbf{u}}\,\dd x
+\theta\int_{\Omega}T(\vr)(\mathbf{u}{\cdot}\nabla\mathbf{u})\cdot\bar{\mathbf{u}}\,\dd x
+\delta\theta\int_{\Omega}\vr\nabla\mathbf{u}{:}\nabla\bar{\mathbf{u}}\,\dd x\\&
-\delta\theta\int_{\{0\leq\vr\leq\vr_M\}}\vr G_{R,\delta}'(\vr)\,\dd x
-\theta
\int_{\Omega}\mathbb{S}(\nabla\bar{\mathbf{u}}):\nabla(\mathbf{u}-\bar{\mathbf{u}})\,\dd x
+\delta\theta C(M,p),
\end{aligned}
\end{equation}
for certain $\omega > 0$, where $C(M,p)$ depends only on $M$ and
the structural properties of the function $p$. In particular, the estimate
is independent of $R$ and $\delta$. We focus now on the difficult terms that should
be controlled in \eqref{eq:exists16}. First, we observe that
there exists a constant \(C_1>0\) independent of \(R\) and \(\delta\) such that
\begin{equation*}
\begin{aligned}
\intO{ T(\vr) \vu \cdot \Grad \vu \cdot \bar{\vu} }
&\leq \bar{\vr} \|\Grad
\vu \|_{\L^2(\Omega)} \| \vu \|_{\L^4(\Omega)} \| \bar{\vu} \|_{\L^4(\Omega)}
\leq  C_1 \| \bar{\vu} \|_{\L^4(\Omega)} \| \vu \|_{\W^{1,2}(\Omega)}^2
\\&\leq
2C_1\norm[\L^4(\Omega)]{\bar{\vu}}(\norm[\W^{1,2}(\Omega)]{\vu-\bar\vu}^2+
\norm[\W^{1,2}(\Omega)]{\bar{\vu}}).
\end{aligned}
\end{equation*}
as long as \(N=2,3\).
In accordance with Lemma \ref{aL1} established below, the extension $\bar{\vu}$ can be chosen
by such a way that $\|\bar{\vu} \|_{\L^4(\Omega)}$
is arbitrarily small provided $\bar{\vu} \in \W^{1,p}(\Omega)$ with $p > N$.
We will therefore consider such an extension supposing
\begin{equation}
\label{E15}
\bar{\vu} \in\W^{1,p}(\Omega)\quad\text{with}\quad p > N\quad\text{and}\quad
\bar{\vu} \cdot \vc{n}= 0\quad\text{on}\quad\partial \Omega.
\end{equation}
Secondly, we have
\begin{equation}
\label{E15b}
\begin{aligned}
&-\delta\theta\int_{\Omega}(\nabla\vr\cdot\nabla\vu)\cdot\bar{\vu}\,\dd x
\leq \frac{\theta\delta^{\frac32+r}}2\norm[\L^{\infty}(\Omega)]{\bar{\vu}}
\norm[\L^2(\Omega)]{\nabla\vr}^2
\\&+\theta\delta^{\frac12-r}
\norm[\L^{\infty}(\Omega)]{\bar\vu}(\norm[\W^{1,2}(\Omega)]{\vu-\bar\vu}^2+\norm[\L^2(\Omega)]{
\nabla\bar\vu}^2)
\end{aligned}
\end{equation}
with \(r\in(0,\frac12)\). The first term on the right hand side is controlled by
\(\theta\delta^{\frac32}\int_{\Omega}\frac{\abs{\nabla\vr}^2}{T(\vr)}\dd x\)
(take \(\delta\) small enough)
while the second one is controlled by \(\frac1{C_P}\norm[\W^{1,2}(\Omega)]{\vu-\bar\vu}^2\).
On the other hand, we have
\begin{equation*}
\delta\theta\int_{\Omega}\vr\nabla\vu:\nabla\bar\vu\dd x
\leq
\frac{\delta\theta}2\int_{\Omega}\vr\abs{\nabla\vu}^2\,\dd x
+\frac{\delta\theta}4\norm[\L^4(\Omega)]{\nabla\bar\vu}^4+
\frac{\delta\theta}4\norm[\L^2(\Omega)]{\vr}^2.
\end{equation*}
Now from the definition of \(G_{R,\delta}'\), we have
\begin{equation}
\label{GR}
G_{R,\delta}'(\vr)\geq \frac1{\bar\vr}
\bigl(p_R(\vr)-p_R(\vr_M)+\sqrt{\delta}(\vr-\vr_M)\bigr)\quad\text{for}\quad\vr\geq\vr_M.
\end{equation}
The expression of \(p_R\) leads to
\begin{equation*}
G_{R,\delta}'(\vr)\geq \frac{\eta\vr}{\bar\vr}\quad\text{for}\quad\vr\geq\bar\vr,
\end{equation*}
with \(\eta>0\) and \(R\) taken large enough. We have
\begin{equation*}
\norm[\L^2(\Omega)]{\vr}^2=
\int_{\vr\leq\bar\vr}\vr^2\,\dd x+\int_{\vr>\bar\vr}\vr^2\,\dd x
\leq\bar{\vr}^2\abs{\Omega}+\frac{\bar\vr}{\eta}\int_{\vr>\bar\vr}\vr G_R'(\vr)\,\dd x.
\end{equation*}
We conclude applying Lemma \ref{fix_point} that a solution \((\vu,\vr)\)
to problem \eqref{weak_penalty}
satisfies as well the following inequality:
\begin{equation}
\label{E16}
\begin{aligned}
&
\| \vu - \bar{\vu} \|^2_{\W^{1,2}(\Omega)}
+ \delta \intO{\vr| \Grad \vu|^2}
+ \delta\intO{\Bigl(\frac{p'_R(\vr)+\sqrt{\delta}}{T(\vr)}\Bigr) |\Grad \vr|^2}
\\&
+ \delta\int_{\{\vr\geq\vr_M\}}{G_{R,\delta}' (\vr) \vr}\,\dd x \leq C_2,
\end{aligned}
\end{equation}
where $C_2>0$ is a constant independent on $d$, $R$ and $\delta$. 
This proves the uniform bound of $\mathbf{v_\theta} = \mathbf{v}$, solution of 
\eqref{eq:exists22}, 
and thus the existence result to problem \eqref{weak_penalty} by 
using Lemma \ref{fix_point}.
Besides, the standard elliptic theory implies that there exists a constant \(C_{\delta}>0\)
depending only on \(\delta\) such that
\begin{equation}
\label{E16b}
\| \vr \|_{\W^{2,2}(\Omega)} \leq C_{\delta}.
\end{equation}

\subsection{Limit $d \to +\infty$}
\label{d}

The passage to the limit as \(d\) tends to \(+\infty\) in \eqref{weak_penalty}
follows from estimates \eqref{E16} and \eqref{E16b} in \eqref{weak_penalty},
the verification is let to the reader. Then, we may conclude that the
limit solution, denoted once again by \((\vu,\vr)\), satisfies the following system:
\begin{subequations}
\label{weak_penalty0}
\begin{align}
\notag
&\text{Find }\mathbf{u}= \mathbf{v}+\bar{\mathbf{u}}, \mathbf{v}\in\W_0^{1,2}
(\Omega)\text{ and }\vr\in\W^{2,2}(\Omega)\text{ such that for all }
\mathbf{w}\in\W_0^{1,2}(\Omega),\\&
\label{weak_penalty01}
\int_{\Omega}\mathbb{S}(\nabla\mathbf{u})\cdot\nabla\mathbf{w}\,\dd x=
\int_{\Omega}
(T(\vr){\mathbf{u}}\otimes{\mathbf{u}}):\nabla\mathbf{w}\,\dd x\\&\notag
+\int_{\Omega}(p_R(\vr)+\sqrt{\delta}\vr)\dive\mathbf{w}\,\dd x
-\delta\int_{\Omega}\nabla(\vr{\mathbf{u}})\cdot\nabla\mathbf{w}\,\dd x
-\int_{\Omega}\delta\vr{\mathbf{u}}\mathbf{w}\,\dd x,\\&
\label{weak_penalty02}
\delta\vr-\delta\Delta
\vr+\dive(T(\vr)\mathbf{u})=\delta\vr_M
\quad\text{in}\quad\Omega,\\&
\label{weak_bound_cond0}
\mathbf{u}=\bar{\mathbf{u}}\quad\text{and}\quad
\nabla\vr\cdot\mathbf{n}=0\quad\text{on}\quad\partial\Omega,\\&
\label{totalmassweak0}
\int_{\Omega}\vr\,\dd x=M.
\end{align}
\end{subequations}

\subsection{Limit $R \to +\infty$}
\label{R}

In the present section, our goal consists to perform the asymptotic limit for $R$
tending to $+\infty$ in \eqref{weak_penalty01}--\eqref{weak_bound_cond0}.
Here the family of solutions of \eqref{weak_penalty01}--\eqref{weak_bound_cond0}
is denoted by $(\vuR,\vrR)_{R >\frac1{\bar{\vr}}}$.
Now, keeping $\delta > 0$ fixed, the compactness provided by the
artificial viscosity approximation plays a crucial role as well as
the inequalities \eqref{E16} and \eqref{E16b} which
remain valid when \((\vu,\vr)\) is replaced by \((\vu_R,\vr_R)\).
Since \eqref{E16b} holds and $N=2,3$, we may assume
\begin{equation}
\label{E17}
\vrR \to \vr \quad\text{in}\quad\C^0(\bar{\Omega}) \quad\text{and}\quad\vr \geq 0.
\end{equation}
Next we observe that the limit density satisfies
\begin{equation}
 \label{E18}
0 \leq \vr < \bar{\vr} \quad\text{a.e. in}\quad\bar{\Omega}.
\end{equation}
Indeed, let us take for any \(Z\) such that \(\vr_M<Z<\bar{\vr}\):
\begin{equation*}
F_Z(\vr) = \min \Bigl(1, \frac{2}{ \bar{\vr} - Z} (\vr - Z )^+ \Bigr).
\end{equation*}
We proceed by contradiction; let us suppose that \(\text{meas}\{\vr\geq \bar{\vr}\}>0\).
Then for any \(Z\in]\vr_M,\bar{\vr}[\), we get
\begin{equation}
\label{E18b}
\int_{\Omega}F_Z(\vr)\,\dd x\geq\int_{\{\vr\geq\bar{\vr}\}}1\,\dd x>0.
\end{equation}
Now we have
\begin{equation*}
\begin{aligned}
\int_{\Omega} F_Z(\vrR)\dd x& \leq
\frac{1}{G'_{R, \delta}(Z)}\int_{\{ \vrR \geq Z \}} G'_{R, \delta}(Z) \dd x
\\&\leq  \frac{1}{Z G'_{R, \delta}(Z)} \int_{\{ \vrR \geq Z \}} \vrR G'_{R, \delta}(\vrR) \dd x.
\end{aligned}
\end{equation*}
In view of the uniform bounds established in \eqref{E16}, we may let $R$ tends to $+\infty$
obtaining that there exists a constant \(C_\delta>0\) independent 
on \(R\) such that
\begin{equation*}
\int_{\Omega} F_Z(\vr)\dd x \leq \frac{C_\delta}{ Z G_{\delta}'(Z)}.
\end{equation*}
By using \eqref{GR} (remains still valid if
\(G'_{R,\delta}(Z)\) is replaced by \(G'_{\delta}(Z)\) and \(p_R\) by $p$), 
we deduce that
\(G'_{\delta}(Z)\) tends to \(+\infty\) as \(z\)  tends to \(\bar{\vr}\). According to the above inequality,
we may conclude that this fact contradicts the estimate \eqref{E18b}.
Furthermore, we have \(0\leq\vr\leq\bar{\vr}_{\delta}\) where
\(\bar{\vr}_{\delta}\) is a constant such that \(0<\bar{\vr}_{\delta}<\bar{\vr}\).

Finally, we need to exhibit bounds on the pressure. To this aim, we remark that
\eqref{weak_penalty01} and the estimate \eqref{E16} lead to
\begin{equation}
\label{E19}
\| \Grad p_R(\vrR) \|_{\W^{-1,2}(\Omega)} \leq c(\delta).
\end{equation}
In addition, as $\int_{\Omega}\vrR G'_{R, \delta}(\vrR)\dd x$ is uniformly bounded
and using \eqref{GR}, we deduce that
\begin{equation}
\label{E20}
\| p_R (\vrR) \|_{\L^1(\Omega)} \leq c(\delta),
\end{equation}
where \(c(\delta)>0\) is a constant depending on \(\delta\).\\
Using \eqref{E19} and \eqref{E20}, there exists a limit denoted by \(\Ov{p(\vr)}\) such that
\begin{equation*}
p_R(\vrR) \to \Ov{p(\vr)} \quad{in}\quad\L^2(\Omega)\quad\text{weak}.
\end{equation*}
In view of \eqref{E17} and \eqref{E18}, we have
\begin{equation*}
\Ov{p(\vr)} = p(\vr).
\end{equation*}
Letting $R$ tends to $+\infty$ in \eqref{weak_penalty01}--\eqref{weak_bound_cond0},
we obtain the following system of equations:
\begin{subequations}
\label{E02}
\begin{align}
\label{E21}
&\Div (\vr \vu) - \delta \Del \vr + \delta \vr =\delta \vr_M\quad\text{in}\quad
\Omega,\\
\label{E22}
&\Div (\vr \vu \otimes \vu ) + \Grad p (\vr) + \sqrt{\delta} \Grad \vr =
\Div (\mathbb{S}(\Grad \vu)) + \delta \Del (\vr \vu) -  \delta \vr \vu
\quad\text{in}\quad\Omega,\\
\label{E23}
&\Grad \vr \cdot \vc{n} = 0\quad\text{and}\quad\vc{u}= \bar{\vu}
\quad\text{on}\quad\partial \Omega,\\&
\label{E23b}
\int_{\Omega}\vr\,\dd x=M,
\end{align}
\end{subequations}

\subsection{Limit $\delta \to 0$}
\label{limit_delta}

Our ultimate goal is to perform the limit $\delta$ tending to $0$ in system \eqref{E02}.
Let $(\vud,\vrd)$ be the approximate solutions of problem \eqref{E02}.
In addition to the estimates \eqref{E16} that hold uniformly for $\delta$ tending to $0$, we
have at hand a rather strong bound on the density, namely,
\begin{equation}
 \label{E24}
0 \leq \vrd < \bar{\vr}\quad\text{a.e. in}\quad\Omega.
\end{equation}
Consequently, passing to suitable subsequences as the case may be,
we may suppose
\begin{equation}
\label{E25}
\vrd \rightharpoonup \vr \quad\text{in}\quad\L^\infty(\Omega)\quad\text{weak-*}
\quad\text{and}\quad\vud \rightharpoonup
\vu \quad\text{in}\quad\W^{1,2}(\Omega)\quad\text{weak},
\end{equation}
where
\begin{equation*}
0 \leq \vr \leq \bar{\vr},\quad \int_{\Omega}{\vr}\dd x = M
\quad\text{and}\quad\vu=\bar{\vu}\quad\text{on}\quad\partial \Omega.
\end{equation*}

\subsubsection{A preliminary result}
\label{Prem_result}

\begin{lemma}
\label{Lemma_prem}
Assume that $p: [0,\bar{\vr}) \to \R$ belongs to $\C[0,\bar \vr)$, is non decreasing and
$ \lim_{\vr \to \bar{\vr}} p(\vr)= \infty$. Let
$\vrd \in\L^\infty(\Omega)$ be a sequence such that
\begin{equation*}
0 \leq \vrd < \Ov{ \vr } \quad\text{and}\quad
\|p(\vrd)\|_{\L^\beta(\Omega)}  \leq C,
\end{equation*}
with $C$ a constant independent on $\delta$ and $\beta \in (1,+\infty)$. Suppose that
\begin{equation*}
\begin{aligned}
&\vrd \rightharpoonup \vr\quad\text{in}\quad\L^\infty(\Omega)\quad\text{weak-*},\\
&p(\vrd) \rightharpoonup \Ov{p(\vr)}\quad\text{and}\quad
 \vrd p(\vrd) \rightharpoonup \Ov{\vr p(\vr)}\quad\text{in}\quad\L^\beta(\Omega)\quad\text{weak}.
 \end{aligned}
 \end{equation*}
Then we have
\begin{subequations}
\begin{align}
\label{strictrhob}
&0 \leq \vr < \bar{\vr} \quad\text{a.e. in}\quad\Omega,\\
\label{identity}
&\text{If, in addition, } \; \Ov{\vr p(\vr)}=\vr \Ov{ p(\vr)}\quad\text{then}\quad\Ov{ p(\vr)} = p(\vr).
\end{align}
\end{subequations}
\end{lemma}	

\begin{proof}
Clearly, $0 \leq \vr \leq \bar{\vr}$.
Let us show the strict inequality in \eqref{strictrhob} by contradiction. To this aim,
we denote by
\begin{equation*}
A\eqldef\{x \in \Omega: \vr(x)= \bar{\vr}\}
\end{equation*}
and we suppose that $|A| >0$.
Taking $ \varphi(x)=1_A$, the indicator function on \(A\), as a test function in the
weak-* convergence of $ \vrd$ to $\vr$. Since $0 \leq \vrd <  \bar{\vr}$,
$\int_A (\vrd-\bar{\vr})\dd x$ tends to \(0\) that is $ \vrd$ converges
strongly to $\bar{\vr}$ in $\L^1(A)$.
We deduce (up to a subsequence of $\delta$) that
\begin{equation*}
\vrd(x) \to \bar{\vr}\quad\text{a.e. in}\quad A,
\end{equation*}
which implies
\begin{equation*}
\int_A(p(\vrd(x)))^\beta \dd x \to +\infty.
\end{equation*}
This contradicts the hypothesis and proves \eqref{strictrhob}.

Let us now establish \eqref{identity}.
For any $\eta > 0$ small enough, we consider the continuous function
$p_\eta \in \C^0([-\frac \eta 2,\bar{\vr} ))$ given by
\begin{equation*}
p_\eta(z)\eqldef
\begin{cases}
p(z) \text{ for } z \in (0,\bar{\vr}),\\
p(0) \text{ for } z \in [-\frac \eta 2,0).
\end{cases}
\end{equation*}
Furthermore, it is convenient to define
\begin{equation*}
\Omega_\eta \eqldef\{ x \in \Omega: 0 \leq \vr(x) \leq  \bar{\vr}-\eta\}.
\end{equation*}
Now for an arbitrary $\zeta$ belonging to
$\L^\infty(\Omega)$ with $-\frac \eta 2 \leq \zeta \leq \bar{\vr}-\frac \eta 2$,
we have immediately that $ (p_\eta(\vrd)-p_\eta(\zeta))(\vrd-\zeta) \geq 0$.
Therefore passing to the limit as $\delta$ tends to $0$ and using the
assumption $\Ov{\vr p(\vr)}=\vr \Ov{p(\vr)}$, we conclude that
$(\Ov{p(\vr)}-p_\eta(\zeta))(\vr-\zeta)\geq 0$ in $\Omega$.
Integrating over $\Omega_\eta$ and taking $\zeta=\vr+\epsilon \psi$
with $\psi\in\L^\infty(\Omega)$ arbitrary
and $\epsilon > 0$ a small parameter and passing to the limit as $\epsilon$ tends to \(0\), we find
\begin{equation*}
\int_{\Omega_\eta } (\Ov{p(\vr)}-p(\vr))\psi \dd x\geq 0.
\end{equation*}
Therefore we obtain $ \Ov{p(\vr)}=p(\vr)$ on $\Omega_\eta$ for any $\eta > 0$ small enough.
Since \eqref{strictrhob} holds, it follows that
$\Omega=\cup_{\eta > 0}\Omega_\eta \cup A$
which concludes the proof.
\end{proof}
	
\subsubsection{Limit in the field equations}
\label{Limit_field}

Revisiting \eqref{E16}, we see that
\begin{equation*}
\delta^{3/2} \| \Grad \vrd \|^2_{\L^2(\Omega)} \leq c.
\end{equation*}
Consequently, in view of \eqref{E25} and compactness of the embedding
\(\W^{1,2}(\Omega) \hookrightarrow\hookrightarrow \L^q(\Omega)\) with \(q\in[1,6)\)
if  \(N=3\) and \(q\geq 1\) arbitrary finite if  \(N=2\),
it is easy to perform the limit $\delta \to 0$ in \eqref{E21} obtaining
\begin{equation}
\label{E26}
\forall\varphi \in \C^1(\bar{\Omega}):\,
\int_{\Omega}\vr \vu \cdot \Grad \varphi\dd x= 0
\text{ which implies that }\Div (\vr \vu) = 0\text{ in }\mathcal{D}'(\Omega).
\end{equation}
In addition, as $\vr \in\L^\infty(\Omega)$ and $\vu \in\W^{1,2}(\Omega)$,
we may apply the regularizing
technique of DiPerna and Lions \cite{DL} to deduce the renormalized
version of \eqref{E26}, namely, we have
\begin{equation}
\label{E27}
\forall\varphi\in\C^1(\bar{\Omega}):\,
\int_{\Omega}\bigl(b(\vr) \vu \cdot \Grad \varphi+(b(\vr)-b'(\vr) \vr)
(\Div\vu)\varphi\bigr)\dd x=0,
\end{equation}
and any continuously differentiable $b$.
In order to perform the limit in \eqref{E22}, we have to control the pressure.
To this end, we multiply \eqref{E22} by $\vc{B}[\vrd]$,
\begin{equation*}
\Div (\vc{B}[\vrd]) = \vrd-\vr_M, \ \vc{B}|_{\partial \Omega} = 0,
\end{equation*}
where $\vc{B}$ is the Bogovskii operator introduced in  \ref{a}.
As $\vc{B}$ vanishes on the boundary $\partial \Omega$, we may integrate by parts obtaining
\begin{equation}
\label{E28}
\begin{aligned}
&\int_{\Omega}
\bigl(p(\vrd) + \sqrt{\delta} \vrd \bigr) \vrd\,\dd x =
\int_{\Omega}\bigl(p (\vrd) + \sqrt{\delta} \vrd\bigr)\vr_M\,\dd x\\&
- \int_{\Omega}\vrd \vud \otimes \vud : \Grad \vc{B}[\vrd]\,\dd x  +
\int_{\Omega}\mathbb{S}(\Grad \vud) : \Grad \vc{B} [\vrd]\,\dd x\\& +
\delta \int_{\Omega} \Grad (\vrd \vud) \cdot \Grad \vc{B}[\vrd]\,\dd x 
+\delta \int_{\Omega}\vrd \vud
\cdot \vc{B}[\vrd]\,\dd x.
\end{aligned}
\end{equation}
The Bogovskii operator $\vc{B} : \L^q(\Omega) \rightarrow \W^{1,q}_0(\Omega)$,
$1 < q < \infty$, is bounded; whence we may use the uniform bounds \eqref{E16} and \eqref{E24}
to find that
\begin{equation*}
\int_{\Omega}p(\vrd) \vrd\,\dd x \leq c + \int_{\Omega}p(\vrd) \vr_M\,\dd x.
\end{equation*}
Now, similarly to the previous part, we choose $\lambda > 1$ such that
\(\lambda \vr_M < \bar{\vr}\).
Accordingly,
\begin{equation*}
\begin{aligned}
\int_{\Omega}p(\vrd) \vrd &\leq c +
\int_{\Omega}p(\vrd) \vr_M\,\dd x \\&\leq c +
\int_{\{ \vr_{\delta} \leq \lambda \vr_M \}}p(\vrd) \vr_M
\,\dd x + \int_{\{\vr_{\delta} > \lambda \vr_M \}}p(\vrd) \vr_M\,\dd x
\\&\leq c + M p (\lambda \vr_M) +
\frac{1}{\lambda} \int_{\Omega}\vrd p(\vrd)\,\dd x;
\end{aligned}
\end{equation*}
whence
\begin{equation}
\label{E29}
\int_{\Omega}\vrd p(\vrd)\,\dd x \leq c.
\end{equation}
According to \eqref{E28} and \eqref{E29}, we get
\begin{equation}
\label{E29b}
\int_{\Omega}p(\vrd)\,\dd x \leq c.
\end{equation}
Now, we repeat the same procedure with the multiplier $\vc{B}[p^\alpha (\vrd)]$,
where $\alpha > 0$ will be fixed below. Similarly to \eqref{E28}, we have
\begin{equation}
\label{E30}
\begin{aligned}
&\int_{\Omega}\bigl(p(\vrd) + \sqrt{\delta} \vrd \bigr)
p^\alpha(\vrd)\,\dd x = \int_{\Omega} \bigl( p (\vrd) + \sqrt{\delta} \vrd \bigr)
\int_{\Omega}p^\alpha (\vrd)\,\dd x\frac1{\abs{\Omega}}\\
&- \int_{\Omega}\vrd \vud \otimes \vud : \Grad \vc{B}[p^\alpha(\vrd)] \,\dd x+
\int_{\Omega}
\mathbb{S}(\Grad \vud) : \Grad \vc{B} [p^\alpha(\vrd)] \,\dd x
\\ &+ \delta \int_{\Omega}\Grad (\vr \vud) \cdot \Grad \vc{B}[p^\alpha(\vrd)]\,\dd x
+\delta \int_{\Omega} \vr \vud
\cdot \vc{B}[p^\alpha(\vrd)]\,\dd x.
\end{aligned}
\end{equation}
By virtue of \eqref{E29b}, we get
\begin{equation*}
\forall \alpha\in (0,1]:\,
\| p^\alpha (\vrd) \|_{\L^{\frac{1}{\alpha}} (\Omega)} \leq c(\alpha),
\end{equation*}
therefore all integrals on the right--hand side of \eqref{E30}
remain bounded uniformly for $\delta$ tending to $0$ for a suitably small $\alpha > 0$.
Accordingly
\begin{equation}
\label{E31}
\| p(\vrd ) \|_{\L^{\alpha + 1}(\Omega)} \leq c\quad\text{for a certain}\quad\alpha > 0,
\end{equation}
and
\begin{equation}
\label{E32}
p(\vrd) \rightharpoonup \Ov{p(\vr)}\quad\text{in}\quad\L^{\alpha + 1}(\Omega)\quad\text{weak}.
\end{equation}
Letting $\delta$ tends to $0$ in \eqref{E22}, we obtain
\begin{equation}
\label{E33}
\int_{\Omega}
\bigl(\vr\vu\otimes\vu : \Grad\varphi + \Ov{p(\vr)} \bigr) \Div\varphi\,\dd x =
\int_{\Omega}\mathbb{S}(\Grad \vu):\Grad \varphi\,\dd x
\end{equation}
for any $\varphi\in\C^1_c(\Omega)$, or, equivalently,
\begin{equation*}
\Div (\vr \vu \otimes \vu) + \Grad \Ov{p(\vr)} = \Div(\mathbb{S}(\Grad \vu))\quad\text{in}
\quad \mathcal{D}'(\Omega).
\end{equation*}

\subsubsection{Compactness of the density (pressure)}
\label{pres}

The existence proof will be complete as soon as we show that
\begin{equation}
\label{E34}
\Ov{p(\vr)} = p(\vr) \quad\text{a.e. in}\quad \Omega.
\end{equation}
To this aim, we employ the nowadays standard method based on the weak continuity
of the effective viscous flux developed by Lions \cite{LI4}.
Our goal consists in showing that
\begin{equation}
\label{E35}
\lim_{\delta \to 0} \int_{\Omega}p(\vrd) \vrd\,\dd x = \int_{\Omega}\Ov{p(\vr)} \vr\,\dd x.
\end{equation}
and applying Lemma \ref{Lemma_prem} to deduce \eqref{E34}.

To prove \eqref{E35}, we first "renormalize" equation \eqref{E21} to get
\begin{equation*}
\Div (b(\vrd) \vud) + (b'(\vrd) \vrd - b(\vrd))  \Div\vud = \delta \Del \vrd b'(\vrd) - \delta \vrd b'(\vrd) +
\delta \vr_M b'(\vrd).
\end{equation*}
Thus, integrating by parts the above expression, we obtain
\begin{equation*}
\int_{\Omega}(b'(\vrd) \vrd - b(\vrd))  \Div \vud\,\dd x = -\delta\int_{\Omega}
b''(\vrd) |\Grad \vrd |^2\,\dd x + \delta \int_{\Omega}
(\vr_M - \vrd) b'(\vrd)\,\dd x.
\end{equation*}
If $b$ is convex, we have
\begin{equation*}
\int_{\Omega}(b'(\vrd) \vrd - b(\vrd))  \Div \vud\,\dd x \leq
\delta M \max_{ \vr \in [0, \bar{\vr}] } (b'(\vr))
- \delta \int_{\Omega}\vrd b'(\vrd)\,\dd x.
\end{equation*}
In particular, for $b(\vr) = \vr \log(\vr)$ we get
\begin{equation*}
\int_{\Omega} \vrd \Div \vud\,\dd x \leq
\delta M (1+\abs{\log (\bar{\vr})}) - \delta \int_{\Omega} \vrd (1 + \log(\vrd))\,\dd x.
\end{equation*}
Finally, letting $\delta \to 0$, we may conclude that
\begin{equation}
\label{E36}
\limsup_{\delta \to 0} \int_{\Omega}\vrd \Div\vud\,\dd x \leq 0.
\end{equation}

The next step consists in multiplying \eqref{E22} by
\(\varphi \Grad \Del^{-1} [\varphi \vrd]\) where $\Del^{-1}$ denotes the inverse of
 the Laplacian on $\Er^3$, specifically a pseudodifferential operator with Fourier's
 symbol $-\frac{1}{|\xi|^2}$ and \(\varphi \in \C^{\infty}_c(\Omega)\).
After a bit tedious but straightforward computation, we obtain
\begin{equation}
\label{E37}
\begin{aligned}
&\intRN{ \varphi^2 p(\vrd) \vrd } - \intRN{\varphi \mathbb{S}(\Grad \vud) : \Grad \Del^{-1}
\Grad [\varphi \vrd] }\\
&+ \intRN{ \varphi (\vrd \vud \otimes \vud ) : \Grad \Del^{-1} \Grad  [\varphi \vrd] }\\&=
- \intRN{ (\vrd \vud \otimes \vud )\cdot \Grad \varphi \cdot \Grad \Del^{-1} [\varphi \vrd] }
\\&+ \intRN{ \mathbb{S}(\Grad \vud) \cdot \Grad \varphi \cdot \Grad \Del^{-1} [\varphi \vrd] }
- \intRN{p (\vrd) \Grad \varphi \cdot \Grad \Del^{-1} [\varphi \vrd] }
\\&- \sqrt{\delta} \intRN{ \vrd
\Grad \varphi \cdot \Grad \Del^{-1} [\varphi \vrd] } - \sqrt{\delta} \intRN{
\varphi^2 \vrd^2 }\\  &
+\delta \intRN{ \Grad (\vr \vu) \cdot \Grad \varphi \cdot \Grad
\Del^{-1} [\varphi \vrd]} + \delta \intRN{ \vrd \vud
\varphi \Grad \Del^{-1} [\varphi \vrd]}\\&
+\delta\intRN{\varphi\nabla(\vr\bfu):\Grad
\Del^{-1}\nabla [\varphi \vrd]} .
\end{aligned}
\end{equation}
Similarly, we consider \(\varphi \Grad \Del^{-1}[\varphi \vr]\)
as a test function in the limit equation \eqref{E33} and using an analogous approach as
we did for \eqref{E37}, we get
\begin{equation}
\label{E38}
\begin{aligned}
&\intRN{ \varphi^2 \Ov{p(\vr)} \vr } -
\intRN{\varphi \mathbb{S}(\Grad \vu) : \Grad \Del^{-1} \Grad [\varphi \vr] }\\
&+ \intRN{ \varphi (\vr \vu \otimes \vu ) : \Grad \Del^{-1} \Grad  [\varphi \vr] }
= - \intRN{ (\vr \vu \otimes \vu )\cdot \Grad \varphi \cdot \Grad \Del^{-1} [\varphi \vr] }
\\&- \intRN{ \mathbb{S}(\Grad \vu) \cdot \Grad \varphi \cdot \Grad \Del^{-1} [\varphi \vr] }
- \intRN{\Ov{p (\vr)} \Grad \varphi \cdot \Grad \Del^{-1} [\varphi \vr] }.
\end{aligned}
\end{equation}
Next observe that
\begin{enumerate}[(i)]

\item all the terms in \eqref{E37} containing \(\delta\) or \(\sqrt{\delta}\)
vanish in the asymptotic limit when $\delta$ tends to $0$,

\item in view of the compactification effect of
$\Grad \Del^{-1} : \L^q(\Er^N) \to \W^{1,q}_{\rm loc}(\Er^N)$, all integrals on the right--hand
side of \eqref{E37} converge to their counterparts in \eqref{E38}. We have shown that
\begin{equation}
\label{E39}
\begin{aligned}
&\intRN{ \varphi^2 p(\vrd) \vrd}
 + \intRN{ \varphi (\vrd \vud \otimes \vud- \mathbb{S}(\Grad \vud)): \Grad
 \Del^{-1} \Grad  [\varphi \vrd]}\\&
\xrightarrow[\delta\rightarrow 0]{}
 \intRN{ \varphi^2 \Ov{p(\vr)} \vr}
+ \intRN{ \varphi (\vr \vu \otimes \vu-\mathbb{S}(\Grad \vu)) : \Grad \Del^{-1} \Grad  [\varphi \vr]}.
\end{aligned}
\end{equation}

\end{enumerate}
Now we have
\begin{equation*}
\intRN{ \varphi (\vrd \vud \otimes \vud ) : \Grad \Del^{-1} \Grad  [\varphi \vrd] } =
\intRN{ \sum_{j=1}^N \varphi u^j_\delta \sum_{i=1}^N
\bigl(\vrd u^i_{\delta} \partial_{x_i} \Del^{-1} [\partial_{x_j} [\varphi \vrd] \bigr)},
\end{equation*}
where
\(\Div (\vrd \vud) = \delta \Del \vrd - \delta \vrd + \delta \vr_M\)
is a precompact subset of \(\W^{-1,2}(\Omega)\).
On the other hand, as ${\bf curl} [ \Grad \Del^{-1} [ \partial_{x_j} (\varphi \vrd) ] $ vanishes,
we may apply the
celebrated Div--Curl lemma of Murat and Tartar
\cite{Mur78CPC,Tar79CCPD} to conclude that
\begin{equation*}
\intRN{ \varphi (\vrd \vud \otimes \vud ) : \Grad \Del^{-1} \Grad  [\varphi \vrd] }
\xrightarrow[\delta\rightarrow 0]{}
\intRN{ \varphi (\vr \vu \otimes \vu ) : \Grad \Del^{-1} \Grad  [\varphi \vr] }.
\end{equation*}
Accordingly, relation \eqref{E39} reduces to
\begin{equation}
\label{E40}
\begin{aligned}
&\intRN{ \varphi^2 p(\vrd) \vrd } - \intRN{\varphi \mathbb{S}(\Grad \vud) :
\Grad \Del^{-1} \Grad [\varphi \vrd]}\\&
\xrightarrow[\delta\rightarrow 0]{}
\intRN{ \varphi^2 \Ov{p(\vr)} \vr} -
\intRN{\varphi \mathbb{S}(\Grad \vu) : \Grad \Del^{-1} \Grad [\varphi \vr]}.
\end{aligned}
\end{equation}
Finally, we check that
\begin{equation*}
\begin{aligned}
&\lim_{\delta \to 0} \intRN{\varphi \mathbb{S}(\Grad \vud) : \Grad \Del^{-1}
\Grad [\varphi \vrd] } - \intRN{\varphi \mathbb{S}(\Grad \vu) : \Grad \Del^{-1} \Grad [\varphi \vr] }\\
 &=
\lim_{\delta \to 0} \intRN{\Grad \Del^{-1} \Grad[\varphi \mathbb{S}(\Grad \vud)] \varphi \vrd} -
 \intRN{\Grad \Del^{-1} \Grad[\varphi \mathbb{S}(\Grad \vu)] \varphi \vr}\\
 &= (\lambda + 2\mu) \intRN{ \varphi^2 \Div (\vud) \vrd } - (\lambda + 2\mu)
 \intRN{ \varphi^2 \Div (\vu) \vr };
\end{aligned}
\end{equation*}
whence, in combination with \eqref{E40}, we obtain
\begin{equation}
\label{E41}
\intRN{ \varphi^2 \bigl(\Ov{ p(\vr) \vr } - \Ov{p(\vr)} \vr \bigr)}
= (\lambda + 2\mu) \intRN{ \varphi^2
\bigl(\Ov{\Div \vu \vr}
- \Div \vu \vr \bigr)}.
\end{equation}
As the limit functions $(\vr,\vu)$ satisfy the renormalized equation of continuity, we have
\begin{equation*}
\intO{\vr \Div \vu} = 0,
\end{equation*}
which, together with \eqref{E36}, yields
\begin{equation*}
\int_{\Omega}\bigl(\Ov{\Div \vu \vr}
- \Div \vu \vr \bigr)\dd x \leq 0.
\end{equation*}
Moreover, since relation \eqref{E41} is satisfies for any
$\varphi \in \C^{\infty}_c(\Omega)$, we easily deduce that
\begin{equation*}
\int_{\Omega} \bigl(\Ov{p(\vr) \vr } - \Ov{p(\vr)} \vr \bigr)\dd x \leq 0.
\end{equation*}
However, as $p$ is non--decreasing, we have
\begin{equation*}
\Ov{ p(\vr) \vr } - \Ov{p(\vr)} \vr \geq 0;
\end{equation*}
whence \(\Ov{ p(\vr) \vr } = \Ov{p(\vr)} \vr\) a.e. in  \(\Omega\).
In particular, we have shown \eqref{E35} and therefore \eqref{E34} holds.

\section{Justification of the Reynolds system}
\label{J}

We consider now the following situation: 
\begin{equation}
\label{J1}
\Omega=Q_\ep\eqldef
\bigl{\{} x = (x_h,z):\, x_h \in \mathcal{T}^D,\,0 < z < \ep h(x_h),\, h \in
\C^{2 + \nu}(\mathcal{T}^D) \bigr{\}},
\end{equation}
where $\mathcal{T}^D$ is the $D$--dimensional torus, $D=1,2$ and \(\ep > 0\).
In other words, the domain is periodic in the "horizontal" variable
$x_h$ and bounded above in the "vertical" variable $z$ by a graph of a smooth function.
We consider a rescaled version of the compressible Navier--Stokes system,
\begin{subequations}
\label{J_23}
\begin{align}
\label{J2}
&\Div(\vr_\ep \vu_\ep)=0\quad\text{in}\quad Q_{\ep},\\
\label{J3}
&\Div(\vr_\ep \vu_\ep \otimes \vu_\ep) + \frac{1}{\ep^2} \Grad p(\vr_\ep) =
\Div (\mathbb{S}(\Grad \vu_\ep))
\quad\text{in}\quad Q_{\ep},
\end{align}
\end{subequations}
supplemented with the boundary conditions
\begin{equation}
\label{J5}
\vu_\ep|_{\partial \Omega} = \bar{\vu}\eqldef
\begin{cases}
(\vc{s},0)&\text{if}\quad z = 0,\\
0&\text{if}\quad z = \ep h(x_h),
\end{cases}
\end{equation}
where \(\vc{s}\) is a constant vector field.
The pressure $p(\vr_\ep)$ is of the same type as in Theorem \ref{ET1}, and
we fix the total mass of the fluid, we get
\begin{equation}
\label{J6}
\intQe{\vr_\ep}=M_\ep > 0 \quad\text{where}\quad
0<\inf_{\ep > 0}\vr_{M}^{\ep}\leq \sup_{\ep > 0}
\vr_{M}^{\ep} < \bar{\vr},
\end{equation}
where \(\vr_{M}^{\ep}\eqldef \frac{M_\ep}{|Q_\ep|}\).
In accordance with Theorem \ref{ET1}, problem \eqref{J_23}--\eqref{J6} admits a
weak solution $(\vre, \vue)$ for any $\ep$ tending to $0$.
Our goal is to study the asymptotic limit as $\ep$ tends to $0$.

\begin{remark} 
\label{ERE1}
As a matter of fact, the geometry of the underlying spatial domain is 
slightly different from that considered in Theorem \ref{ET1}. 
The existence proof, however, may be performed exactly as in Theorem \ref{ET1}.
\end{remark}

\subsection{Uniform bounds}

We choose an extension $\ove$ of the boundary velocity $\bar{\vu}$ such that
\begin{equation}
\label{J7}
\begin{aligned}
&\ove = \bar{\vu}\quad\text{on}\quad\partial \Qme, \quad\Div\ove = 0,\\&
\| \ove \|_{\L^\infty(\Qme)} \leq c,\quad
\| \nabla_{h} \ove \|_{\L^\infty(\Qme)} \leq c,
\quad\| \partial_z \ove \|_{\L^\infty(\Qme)} \leq c \ep^{-1}.
\end{aligned}
\end{equation}
Here and hereafter, the symbols $\nabla_h$, ${\rm div}_h$, $\Delta_h$
denote the differential operators acting on the horizontal variable $x_h \in \mathcal{T}^D$.

\subsubsection{Energy estimates}
\label{enery_estim}

Similarly to the previous section, we use \(\vu_\ep-\bar{\vu}_\ep\) as a test function in
\eqref{J3} obtaining
\begin{equation*}
\begin{aligned}
&\intQe{\Bigl(\vre \vue \otimes \vue : \Grad \vue  + \frac{1}{\ep^2} p (\vre) \Div \vue \Bigr)}
- \intQe{ \vre \vue \otimes \vue : \Grad \ove }
\\
 &=\intQe{ \bigl(\mathbb{S}(\Grad \vue) - \mathbb{S}(\Grad \ove)\bigr): \Grad (\vue - \ove) } +
 \intQe{ \mathbb{S}(\Grad \ove ): \Grad (\vue - \ove) }.
\end{aligned}
\end{equation*}
In addition, as $(\vre, \vue)$ satisfy the renormalized equation of continuity (see \eqref{E27}), we get
\begin{equation*}
\intQe{p (\vre) \Div \vue} = 0,
\end{equation*}
and we have also
\begin{equation*}
\intQe{\vre \vue \otimes \vue : \Grad \vue} = 0,
\end{equation*}
Consequently, we find
\begin{equation}
\label{J8}
\begin{aligned}
&\intQe{ \bigl( \mathbb{S}(\Grad \vue) - \mathbb{S}(\Grad \ove) \bigr): \Grad (\vue - \ove) } \\
& = \intQe{ \vre \vue \otimes \vue : \Grad \ove } - \intQe{ \mathbb{S}(\Grad \ove ): \Grad (\vue - \ove ) }.
\end{aligned}
\end{equation}
Now, we recall the Poincar\' e inequality on $\Qme$,
\begin{equation}
\label{J9}
\| v \|_{\L^2(\Qme)} \leq \ep \| \Grad v \|_{\L^2(\Qme)} \quad\text{whenever}\quad
{v} \in\W^{1,2}(\Qme)\quad\text{and}\quad v|_{z = \ep h(x_h)} = 0,
\end{equation}
and Korn's inequality
\begin{equation}
\label{J10}
\| \Grad \vc{v} \|^2_{\L^2(\Qme)} \leq c \intQe{ \mathbb{S} (\Grad \vc{v} ): \Grad \vc{v} } \quad\text{if}
\quad\vc{v} \in \W^{1,2}_0(\Qme).
\end{equation}
Combining \eqref{J7}--\eqref{J10}, we get
\begin{equation*}
\begin{aligned}
&\frac{1}{\ep^2} \| \vue - \ove \|^2_{\L^2(\Qme)} + \| \Grad \vue - \Grad \ove \|^2_{\L^2(\Qme)} \\ &\leq
c \bigl( \bar{\vr} \| \Grad \ove \|_{\L^\infty(\Qme)} \| \vue \|^2_{\L^2(\Qme)} + \| \Grad \ove \|_{\L^2(\Qme)}
\| \Grad \vue - \Grad \ove \|_{\L^2(\Qme)}  \bigr)\\&
\leq c \bigl( \ep \bar{\vr} \| \Grad \vue \|^2_{\L^2(\Qme)} + \| \Grad \ove \|_{\L^2(\Qme)}
\| \Grad \vue - \Grad \ove \|_{\L^2(\Qme)}  \bigr) \\&\leq c \bigl(\ep \bar{\vr}
\| \Grad \vue - \Grad  \ove \|^2_{\L^2(\Qme)}
+ \ep \bar{\vr} \| \Grad \ove \|_{\L^2(\Qme)}^2 + \| \Grad \ove \|_{\L^2(\Qme)}
\| \Grad \vue - \Grad \ove \|_{\L^2(\Qme)}\bigr).
\end{aligned}
\end{equation*}
Thus, by H\" older's inequality, it follows that
\begin{equation}
\label{J11}
\| \vue - \ove \|^2_{\L^2(\Qme)} + \ep^2 \| \Grad \vue - \Grad \ove
\|^2_{\L^2(\Qme)} \leq c \ep^2 \| \Grad \ove \|^2_{\L^2(\Qme)}.
\end{equation}

\subsubsection{Pressure estimates}

First of all, we estimate the integral mean of the pressure. To this end, we use
the Bogovskii operator $\vc{B}_\ep [\vre]$, specifically,
\begin{equation*}
\Div (\vc{B}_\ep [\vre]) = \vre - \frac{1}{|\Qme|} \intQe{ \vre }
\quad\text{with}\quad \vc{B}_\ep [\vre]=0\quad\text{on}\quad\partial\Qme,
\end{equation*}
as a test function in \eqref{J3}, we obtain
\begin{equation*}
\begin{aligned}
&\intQe{ p(\vre) \vre } = \intQe{ \vr_{M}^{\ep} p(\vre) }
\\&-
\ep^2 \intQe{ \vre \vue \otimes \vue : \Grad \vc{B}_\ep [\vre] }
+ \ep^2 \intQe{ \mathbb{S}(\Grad \vue) : \Grad \vc{B}_\ep [\vre] }.
\end{aligned}
\end{equation*}
Consequently, the H\" older's inequality leads to
\begin{equation}
\label{J12}
\begin{aligned}
&\intQe{ p(\vre) \vre } \leq \intQe{ \vr_{M}^{\ep} p(\vre) } \\&
+ \ep^2 \bigl( \bar{\vr} \| \vue \|^2_{\L^4(\Qme)}
+ \| \Grad \vue \|_{\L^2(\Qme)} \bigr)
 \| \Grad \vc{B}_\ep [\vre] \|_{\L^2(\Qme)}.
\end{aligned}
\end{equation}
On the other hand, we may observe that (see \cite[Lemma 9]{MarPal})
\begin{equation*}
\ep  \| \Grad \vc{B}_\ep [\vre] \|_{\L^2(\Qme)} \leq c \| \vre \|_{\L^2(\Qme)}
\end{equation*}
Then \eqref{J11} and \eqref{J12} give rise to
\begin{equation}
\label{J13}
\begin{aligned}
&\intQe{ p(\vre) \vre } \leq \intQe{ \vr_{M}^{\ep} p(\vre) }\\& + \ep c \bigl( \| \vue \|^2_{\L^4(\Qme)}
+ \| \Grad \ove \|_{\L^2(\Qme)} \bigr)
 \| \vre \|_{\L^2(\Qme)}.
 \end{aligned}
\end{equation}
Similarly to Section \ref{E}, we take
\begin{equation*}
\lambda > 1\quad\text{and}\quad\lambda \vr_{M}^{\ep} < \bar{\vr},
\end{equation*}
and estimate
\begin{equation*}
\begin{aligned}
&\intQe{ \vr_{M}^{\ep} p(\vre) } \leq \int_{\{ \vre < \lambda
\vr_{M}^{\ep} \}}\vr_{M}^{\ep} p(\vre)
\,\dd x + \frac{1}{\lambda} \intQe{ \vre p(\vre) }
\\ &\leq M_\ep p \bigl(\lambda \vr_{M}^{\ep}\bigr)
+ \frac{1}{\lambda} \intQe{ \vre p(\vre) }
\leq c\ep p ( \lambda \vr_{M}^{\ep})+\frac{1}{\lambda} \intQe{ \vre p(\vre) }.
\end{aligned}
\end{equation*}
Finally, we may deduce that \eqref{J13} reduces to
\begin{equation}
\label{J14}
\frac{1}{|\Qme|} \intQe{ p(\vre) \vre } \leq c\bigl(1 +  \bigl( \| \vue \|^2_{\L^4(\Qme)}
+ \| \Grad \ove \|_{\L^2(\Qme)} \bigr)
 \| \vre \|_{\L^2(\Qme)} \bigr).
\end{equation}
On the other hand, again by interpolation, we find
\begin{equation*}
\| v \|_{\L^4(\Qme)} \leq \| v \|^{1/4}_{\L^2(\Qme)} \| v \|_{\L^6(\Qme)}^{3/4}
\leq c \ep^{1/4} \| \Grad v \|_{\L^2(\Qme)}^{1/4} \| \Grad v \|^{3/4}_{\L^2(\Qme)} =
 c \ep^{1/4} \| \Grad v \|_{\L^2(\Qme)},
\end{equation*}
if \(D = 2\) and
\begin{equation*}
\| v \|_{\L^4(\Qme)} \leq \| v \|^{1 - \beta}_{\L^2(\Qme)} \| v \|_{\L^q(\Qme)}^{\beta}
\leq c \ep^{1- \beta} \| \Grad v \|_{\L^2(\Qme) }^{1- \beta}  \| v \|^{\beta}_{\L^q(\Qme)} =
 c \ep^{1 - \beta} \| \Grad v \|_{\L^2(\Qme)}
\end{equation*}
for any \(\beta > \frac{1}{2}\) if \(D = 1\),
whenever $v \in \W^{1,2}(\Qme)$ and $v|_{z = \ep h(x_h)} = 0$.
Indeed $v$ can be extended to be zero for $\ep h(x_h) \leq z < 1$ and, consequently,
we may use the standard Sobolev embedding on the strip
$\mathcal{T}^N \times (0,1)$, namely $\W^{1,2} \hookrightarrow \L^6$
if $N=3$, $\W^{1,2} \hookrightarrow \L^\beta$ for any
$\beta$ finite if $N=2$. In both cases, relation \eqref{J14} together with the energy
bound \eqref{J11}  give rise to
\begin{equation}
\label{J15}
\begin{aligned}
&\frac{1}{|\Qme|} \intQe{ p(\vre) \vre } \leq c\bigl(1 +  \bigl( \| \vue \|^2_{\L^4(\Qme)}
+ \| \Grad \ove \|_{\L^2(\Qme)} \bigr)
 \| \vre \|_{\L^2(\Qme)} \bigr) \\
 &\leq c\bigl( 1 +  \bigr( \ep^{1/2} \| \Grad \vue \|^2_{\L^2(\Qme)}
+ \| \Grad \ove \|_{\L^2(\Qme)} \bigr)
 \| \vre \|_{\L^2(\Qme)} \bigr)
\\ &\leq c\bigl(1 +  \bigl( \ep^{1/2} \| \Grad \ove \|^2_{\L^2(\Qme)}
+ \| \Grad \ove \|_{\L^2(\Qme)} \bigr)
 \| \vre \|_{\L^2(\Qme)} \bigr) \\ &\leq c \bigl(1 + \| \Grad \ove \|_{\L^2(\Qme)}
 \| \vre \|_{\L^2(\Qme)} \bigr)
 \leq \bar{\vr} c.
\end{aligned}
\end{equation}
With \eqref{J15} at hand, we may estimate the pressure $p(\vre)$ in the $\L^2$--norm.
Indeed, repeating the above procedure with $\vc{B}[p(\vre)]$,
\begin{equation*}
\Div \vc{B}_\ep [p(\vre)] = p(\vre) - \frac{1}{|\Qme|} \intQe{ p(\vre) },
\end{equation*}
we get
\begin{equation*}
\begin{aligned}
&\intQe{ |p(\vre)|^2 } = \frac1{\abs{Q_\ep}}\Bigl( \intQe{ p(\vre) } \Bigr)^2
- \ep^2 \intQe{ \vre \vue \otimes \vue : \Grad \vc{B}_\ep [p(\vre)] }\\&
+ \ep^2 \intQe{ \mathbb{S}(\Grad \vue) : \Grad \vc{B}_\ep [p(\vre)] }.
\end{aligned}
\end{equation*}
Thus, exactly as in \eqref{J12}--\eqref{J13}, we have
\begin{equation*}
\intQe{ |p(\vre)|^2 } \leq
\frac1{\abs{Q_\ep}}\Bigl( \intQe{ p(\vre) } \Bigr)^2
+ c \ep \| \Grad \ove \|_{\L^2(\Qme)} \| p(\vre) \|_{\L^2(\Qme)};
\end{equation*}
which implies by using \eqref{J15} that
\begin{equation}
\label{J16}
\begin{aligned}
&\| p(\vre) \|^2_{\L^2(\Qme)} \leq
\frac1{\abs{Q_\ep}}
\Bigl( \intQe{ p(\vre) } \Bigr)^2
+ \ep^2 \| \Grad \ove \|_{\L^2(\Qme)}^2\\&
\leq c\ep \bigl(1 + \ep\| \Grad \ove \|_{\L^2(\Qme)}^2 \bigr).
\end{aligned}
\end{equation}

\subsection{Limit $\ep \to 0$}

In order to perform the limit as $\ep$ tends to $0$,
it is convenient to work on a fixed spatial domain.
Introducing a new vertical variable $Z = \frac{1}{\ep} z$, we get
\begin{equation}
\label{J17}
Q \eqldef \bigl{\{} x = (x_h,Z):\; x_h \in \mathcal{T}^D ,
\ 0 < z < h(x_h),\ h \in \C^{2 + \nu}(\mathcal{T}^1) \bigr{\}}.
\end{equation}
For $\vue = (\vu_{h,\ve}, V_\ve)$, the system \eqref{J_23} written in the new coordinates reads:
\begin{subequations}
\begin{align}
\label{J18}
&{\rm div}_h (\vr \vu_{h,\ve} ) + \frac{1}{\ep} \partial_Z (\vr {V_\ve}) = 0,\\
\label{J19}
&{\rm div}_h (\vr \vu_{h,\ve} \otimes \vu_{h,\ve} ) + \frac{1}{\ep}
\partial_Z (\vr {V_\ve}  \vu_{h,\ve}) + \frac{1}{\ep^2} \nabla_h p(\vr)
\\ &\notag=
\mu \Bigl(\Delta_h \vu_{h,\ve} + \frac{1}{\ep^2} \partial^2_{Z} \vu_{h,\ve} \Bigr)
+ (\lambda + \mu) \nabla_h {\rm div}_h \vu_{h,\ve} + (\lambda + \mu) \frac{1}{\ep}
\nabla_h \partial_Z {V_\ve},\\
\label{J20}
&{\rm div}_h (\vr \vu_{h,\ve} {V_\ve} ) + \frac{1}{\ep} \partial_Z (\vr V_\ve^2)
+ \frac{1}{\ep^3} \partial_Z p(\vr)
\\ &\notag
= \mu \Delta_h V_{\ve} + \frac{1}{\ep^2} \mu \partial^2_Z {V_\ve}
+ \frac{1}{\ep} \partial_Z {\rm div}_h \vu_{h,\ve} +
(\lambda + \mu) \frac{1}{\ep^2} \partial^2_Z {V_\ve}.
\end{align}
\end{subequations}
Thus the uniform bounds \eqref{J11} yield
\begin{subequations}
\label{J21}
\begin{align}
&\ueh \rightharpoonup \uh \quad\text{in}\quad\L^2(Q)\quad\text{weak},
\quad \Ve \rightharpoonup V \quad\text{in}\quad\L^2(Q)\quad\text{weak} ,\\
&\partial_Z \ueh \rightharpoonup \partial_Z \uh \quad\text{in}\quad\L^2(Q)\quad\text{weak},
\\&
\partial_Z \Ve \rightharpoonup \partial_Z V \quad\text{in}\quad\L^2(Q)\quad\text{weak},\\
\ep &\bigl(\| \nabla_h \vue\|_{\L^2(Q)} + \| \nabla_h \Ve \|_{\L^2(Q)}\bigr) \leq c.
\end{align}
\end{subequations}
On the other hand, it follows from \eqref{J6} that
\begin{equation*}
\begin{aligned}
&V=0\quad\text{for}\quad z=0\quad\text{or}\quad z=h,\\&
\vu_h=s\quad\text{for}\quad z=0\quad\text{and}\quad \vu_h=0\quad\text{for}\quad z=h.
\end{aligned}
\end{equation*}
In addition, we have (see \eqref{a4}),
\begin{equation}
\label{J21a}
\sqrt{\ep} \bigl(\| \ueh \|_{\L^4(Q)} + \| \V_{h,\ve}\|_{\L^4(Q)} \bigr) \leq c,
\end{equation}
and by interpolation, we may deduce that
\begin{equation}
\label{J22bis}
\varepsilon^{1/6}\bigl(
\norm[\L^{12/5}(Q)]{\vu_{h,\varepsilon}}+\norm[\L^{12/5}(Q)]{V_{h,\varepsilon}}\bigl)
\leq c.
\end{equation}
Moreover, by virtue of \eqref{J16},
\begin{equation}
\label{J22}
p(\vre) \rightharpoonup \Ov{p(\vr)} \quad\text{in}\quad\L^2(Q)\quad\text{weak},
\end{equation}
and, finally,
\begin{equation}
\label{J23}
\vre \rightharpoonup \vr \quad\text{in}\quad\L^\infty(Q)\quad\text{weak-*}
\quad\text{and}\quad 0 \leq \vr \leq \bar{\vr}.
\end{equation}
Observe that all limits exhibited in \eqref{J21}, \eqref{J22} and \eqref{J23}
hold up to a suitable subsequence.

\subsubsection{Limit in the field equations}

As $(\vre, \ueh, \Ve)$ satisfy \eqref{J19} and \eqref{J20}
in $\mathcal{D}'(Q)$, we easily deduce from \eqref{J21}--\eqref{J23}) that
\begin{equation}
\label{J24}
\forall\varphi \in \C_c^{\infty}(Q):\,
\intQ{\bigl(- \mu \partial_Z \vu_h \cdot \partial_Z \varphi + \Ov{p(\vr)} \partial_y \varphi \bigr)} = 0.
\end{equation}
On the other hand, we have
\begin{equation}
\label{J25}
\forall\varphi \in \DC(Q):\, \intQ{ \Ov{p(\vr)} \partial_Z \varphi } = 0,
\end{equation}
meaning that
\(\Ov{p(\vr)} = \Ov{p(\vr)}(x_h)\) is independent of the vertical coordinate \(Z\).
Finally, we take $\varphi = \varphi (x_h)$ as a test function in \eqref{J18}, we find
\begin{equation*}
\forall\varphi \in \C_c^{\infty}(\mathcal{T}^D):\,
\int_{\mathcal{T}^D} \Bigl(\int_0^{h(x_h)} \Ov{\vr \vu_h } {\rm d}Z \Bigr)
\cdot \nabla_h \varphi \, {\rm d}x_h = 0.
\end{equation*}

\subsubsection{Strong convergence of the pressure}

Our ultimate task is to show strong convergence of the pressure.
First, we find from (\ref{J20})
\begin{equation*}
\begin{aligned}
&\partial_Z p(\vre) =
 - \ep^3 {\rm div}_h (\vre \ueh {V}) - \ep^2 \partial_Z (\vre \Ve^2)
 \\ &+ \mu \ep^3 \Delta_h \Ve + \ep (\lambda + \mu) \partial^2_Z \Ve
+ (\lambda+\mu)\ep^2 \partial_Z {\rm div}_h (\ueh).
\end{aligned}
\end{equation*}
Now we claim that, thanks to the bounds established in \eqref{J21} and \eqref{J22}, we have
\begin{equation}
\label{J26}
\| \partial_z p(\vre) \|_{\W^{-1,2}(Q)} \leq c \ep.
\end{equation}
Indeed the most difficult term reads,
\begin{equation*}
\| \ep^2 \partial_Z (\vre \Ve^2 ) \|_{\W^{-1,2}(Q)} \leq
\| \ep^2 \vre \Ve^2 \|_{\L^2(Q)} \leq \bar{\vr} \ep \| \sqrt{\ep} \Ve \|_{\L^4(Q)}
^2\leq c \ep
\end{equation*}
where we have used \eqref{J21a}. On the other hand,
we employ the method used in Section \ref{pres}, specifically, we consider
\begin{equation*}
\forall\varphi\in\DC(Q):\,
\varphi \Grad \Del^{-1} [\varphi p(\vre)] = \varphi \left[ \nabla_h \Del^{-1} [\varphi p(\vre)] ,
\partial_Z  \Del^{-1} [\varphi p(\vre)] \right],
\end{equation*}
as a test function in \eqref{J19} and \eqref{J20}. Similarly to Section \ref{pres}, we obtain
{\small{
\begin{equation*}
\begin{aligned}
&
\intRNN{ \varphi^2 p(\vre)^2 } = - \intRNN{p (\vre) \Grad \varphi \cdot \Grad \Del^{-1} [\varphi p(\vre)] }
\\&
- \ep^2 \intRNN{ \varphi (\vre \ueh \otimes \ueh)  \nabla_h \Del^{-1} \nabla_h  [\varphi p(\vre)] }
- \ep^2\intRNN{ (\vre \ueh^2 )\cdot \nabla_h \varphi \cdot \nabla_h \Del^{-1} [\varphi p(\vre)] }
\\&
- \ep^3 \intRNN{ \varphi (\vre \ueh \Ve)  \nabla_h \Del^{-1} \partial_Z  [\varphi p(\vre)] }
- \ep^3
\intRNN{ (\vre \ueh \Ve )\cdot \nabla_h \varphi \cdot \partial_Z \Del^{-1} [\varphi p(\vre)] }\\&
- \ep \intRNN{ \varphi (\vre \ueh \Ve)  \partial_Z \Del^{-1} \nabla_h  [\varphi p(\vre)] }
- \ep\intRNN{ (\vre \ueh \Ve )\cdot \partial_Z \varphi \cdot \nabla_h \Del^{-1} [\varphi p(\vre)] }
\\&
- \ep^2 \intRNN{ \varphi (\vre \Ve^2)  \partial_Z \Del^{-1} \partial_Z  [\varphi p(\vre)] }
- \ep^2\intRNN{ (\vre \Ve^2 )\cdot \partial_Z \varphi \cdot \partial_Z \Del^{-1} [\varphi p(\vre)] }
\\&
+ \ep^2 \intRNN{
\beta_1
\varphi {\rm div}_h \ueh {\rm div}_h \Del^{-1}
\nabla_h  [\varphi p(\vre)]}
+ \ep^2 \intRNN{
\beta_1
{\rm div}_h \ueh \nabla_h \varphi \cdot \Del^{-1} \nabla_h
[\varphi p(\vre)]}
\\&
+ \ep \intRNN{
\beta_1
\varphi \partial_Z \Ve {\rm div}_h \Del^{-1}
\nabla_h  [\varphi p(\vre)]}
+ \ep \intRNN{
\beta_1
\partial_Z \Ve \nabla_h \varphi \cdot \Del^{-1}
 \nabla_h  [\varphi p(\vre)]}
 \\&
 + \ep^2 \intRNN{
 \beta_1
 \varphi \partial_Z \ueh \cdot \nabla_h  \Del^{-1}
 \partial_Z  [\varphi p(\vre)]}
 +
\ep^2 \intRNN{
\beta_1
\nabla_h \varphi \cdot \partial_Z \ueh \Del^{-1}
\partial_Z  [\varphi p(\vre)]}
\\&
+ \ep \intRNN{
\beta_2
\varphi \partial_Z \Ve \partial_Z  \Del^{-1}
\partial_Z  [\varphi p(\vre)]}
+\ep \intRNN{
\beta_2
\partial_Z \varphi \partial_Z \Ve \Del^{-1}
 \partial_Z  [\varphi p(\vre)]}\\&
 + \ep^2 \intRNN{ \mu \varphi \nabla_h \ueh \cdot \nabla_h \Del^{-1} \nabla_h [\varphi p(\vre)]}
 +
\ep^2 \intRNN{ \mu \nabla_h \varphi \nabla_h \ueh \cdot \Del^{-1} \nabla_h [\varphi p(\vre)]}
\\&
+ \ep^3 \intRNN{ \mu \varphi \nabla_h \Ve \nabla_h \cdot \Del^{-1} \partial_Z  [\varphi p(\vre)]}
+
\ep^3 \intRNN{ \mu \nabla_h \varphi \cdot \nabla_h \Ve \Del^{-1} \partial_Z  [\varphi p(\vre)]}
\\&
+\intRNN{ \mu \varphi \partial_Z \ueh \partial_Z \Del^{-1} \nabla_h  [\varphi p(\vre)]}
+\intRNN{ \mu \partial_Z \varphi \partial_Z \ueh \cdot \Del^{-1} \nabla_h  [\varphi p(\vre)]},
\end{aligned}
\end{equation*}
}}
where \(\beta_1\eqldef \lambda+\mu\) and \(\beta_2\eqldef \lambda+2\mu\).
Our goal is to let $\ep$ tends to $0$ in the above identity.
According to the following Sobolev embedding:
\(\W^{1,2} \hookrightarrow \L^6\) if \(D = 2\) and
\(\W^{1,2} \hookrightarrow \L^q\) for any finite \(q\) if  \(D=1\), clearly
all terms in the form $\Grad \Del^{-1} [\varphi p(\vre)]$ are uniformly bounded
al least in $\L^6(Q)$. Consequently, it is easy to check
that all integrals in the above inequality
containing these terms and multiplied by some power of
$\ep$ vanish in the limit as $\ep$ tends to $0$ (we also used her \eqref{J22bis}).
Next, we have to check that also the remaining integrals multiplied by some power of
$\ep$ vanish in the asymptotic limit. We focus on the most difficult term,
\begin{equation*}
\begin{aligned}
\ep &\intRNN{\varphi \vre (\ueh \Ve) \cdot \partial_Z \Del^{-1} \nabla_h [ \varphi p(\vre) ]}
\\&= \ep \intRNN{ \vre (\ueh \Ve) \nabla_h \Del^{-1} \partial_Z [ \varphi p(\vre) ]}\\
&= \ep \intRNN{ \vre (\ueh \Ve) \cdot \nabla_h \Del^{-1} [\partial_Z  \varphi p(\vre) ]}
\\&+ \ep \intRNN{ \vre (\ueh \Ve) \cdot \nabla_h \Del^{-1} [ \varphi \partial_Z p(\vre) ]}.
\end{aligned}
\end{equation*}
By virtue of \eqref{J21}, we have
\begin{equation*}
\ep \intRNN{ \vre (\ueh \Ve) \cdot \nabla_h \Del^{-1} [\partial_Z  \varphi p(\vre) ]} \to 0,
\end{equation*}
while, in accordance with \eqref{J21a} and \eqref{J26}, we find
\begin{equation*}
\begin{aligned}
&\Bigl{|} \ep \intRNN{ \vre (\ueh \Ve) \cdot  \nabla_h \Del^{-1} [ \varphi \partial_Z p(\vre) ]}
\Bigr{|} \leq \left\| \ep \vre \ueh \Ve \right\|_{\L^2(Q)}
\| \varphi \partial_Z p(\vre)\|_{\W^{-1,2}(\Er^N)}\\
&\leq c \left\| \sqrt{\ep}  \ueh  \right\|_{\L^4(Q)}\left\| \sqrt{\ep}
\Ve  \right\|_{\L^4(Q)} \left\| \varphi \partial_Z p(\vre) \right\|_{\W^{-1,2}(\Er^N)} \leq c \ep.
\end{aligned}
\end{equation*}
for any $\alpha < 1$.
Consequently, we may let $\ep$ tends to $0$, 
we get
\begin{equation}
\label{J28}
\begin{aligned}
&
\lim_{\ep \to 0} \intRNN{ \varphi^2 p(\vre)^2 }
= - \intRNN{ \Ov{p (\vr)} \Grad \varphi \cdot \Grad \Del^{-1} [\varphi \Ov{p(\vr)}] }
\\
& +
\intRNN{ \mu \varphi \partial_Z \vu_h \cdot \nabla_h \Del^{-1} [\partial_Z \varphi \Ov{p(\vr)}]}
\\&
+
\intRNN{ \mu \partial_Z \varphi \partial_Z \vu_h \cdot \Del^{-1} \nabla_h  [\varphi \Ov{p(\vr)}]}.
\end{aligned}
\end{equation}
Finally, using
\begin{equation*}
\varphi \Grad \Del^{-1} [\varphi \Ov{p(\vr)}]
\end{equation*}
as a \emph{test function} in
\eqref{J24} and \eqref{J25}, then comparing the resulting expression with
\eqref{J28}, we may conclude that
\begin{equation*}
\varphi \in \DC (Q):\,
\lim_{\ep \to 0} \intRNN{ \varphi^2 p(\vre)^2 }
= \intRNN{ \varphi^2 \Ov{p(\vr)}^2 },
\end{equation*}
which immediately yields
\begin{equation*}
p(\vre) \to \Ov{p(\vr)} \quad\text{in}\quad\L^2_{\rm loc}(Q).
\end{equation*}
Under the supplementary hypothesis that \(p'>0\), we get the desired conclusion
\begin{equation}
\label{J29}
\Ov{p(\vr)} = p(\vr).
\end{equation}

\begin{theorem}
\label{ET2}
Let $Q \subset \Er^N$, $N=2,3$, defined as follows:
\begin{equation*}
Q\eqldef \bigl{\{} x = (x_h,z):\;\ x_h \in \mathcal{T}^D ,
\ 0 < z < h(x_h),\ h \in \C^{2 + \nu}(\mathcal{T}^D) \bigr{\}},
\end{equation*}
with \(D =1,2\). Let the pressure be given as $p = p(\vr)$, where
\begin{equation*}
p \in \C^0[0, \bar{\vr}) \cap \C^1(0, \bar{\vr}), \quad p(0) \geq 0,\quad
p'(\vr) > 0 \quad\text{for}\quad\vr > 0, \quad\lim_{\vr \to \bar{\vr}} p(\vr) = +\infty.
\end{equation*}
Let
\begin{equation*}
\frac{M_\ep}{\ep} > 0\quad\text{such that}\quad 0 < \inf_{\ep > 0}
\frac{M_\ep }{|Q|\ep} \leq \sup_{\ep>0} \frac{M_\ep }{|Q|\ep} < \bar{\vr}.
\end{equation*}
Let $(\vre, \ueh, \Ve)$ be a family of weak solutions to the rescaled problem
\eqref{J18}--\eqref{J20}, where the velocities satisfies the boundary conditions \eqref{J5}
and \(\int_{Q}\vr_{\ep}\,\dd x=\frac{M_\ep}{\varepsilon}\)
Then, up to a subsequence, we have
\begin{subequations}
\begin{align}
&\frac{M_{\ep}}{\ep}\to M,\\&
0 \leq \vre \leq \bar{\vr}, \quad\vre \to \vr \quad\text{in}\quad\L^1(Q),\\&
p(\vre) \to p(\vr) \quad\text{in}\quad\L^2_{\rm loc}(Q),\quad
p(\vre) \rightharpoonup p(\vr) \quad\text{in}\quad\L^2(Q)\quad\text{weak},
\\&\ueh \rightharpoonup \vc{u}_h,\quad
\partial_Z \ueh \rightharpoonup \partial_Z \vc{u}_h
\quad\text{in}\quad
\L^2(Q)\quad\text{weak},\\&
\Ve \rightharpoonup V,\quad
\partial_Z V_\ep \rightharpoonup
\partial_Z V \quad\text{in}\quad \L^2(Q)\quad\text{weak},
\end{align}
\end{subequations}
where the limit satisfies
\begin{subequations}
\begin{align}
\label{L1}
&\vr = \vr(x_h),\quad
0 \leq \vr \leq \bar{\vr},\quad p= p(\vr) \in \L^2(\mathcal{T}^D),
\\&
\intQ {\vr} = \int_{\mathcal{T}^D} h \vr \dd x_h = M ,
\\ &\vu_h |_{Z = 0} = \vc{s},\quad \vu_h|_{Z = h(x_h)} = 0,\\&
\label{L2}
{\rm div}_h \Bigl( \int_0^{h(x_h)} \vr \vu_h \ {\rm d}Z \Bigr) = 0,\quad
- \mu \partial^2_Z \vu_h + \nabla_h p(\vr) = 0
\end{align}
\end{subequations}
in $\mathcal{D}'(\Omega)$.
\end{theorem}

\subsection{Uniqueness for the limit problem}
\label{U}

It is interesting to know if the convergence stated in Theorem \ref{ET2}
is unconditional, meaning if the limit problem
\eqref{L2} is uniquely solvable in the class \eqref{L1}.
We show that this is indeed the case at least for $D=1$,
$p'(\vr) > 0$ with  $\vr\in(0, \bar{\vr})$.
We denote $v = \vu_h$ and $x_h = y$ and we assume that \(s>0\).

To this aim,  we first observe that the second equation in \eqref{L2},
and the boundary conditions satisfied $v$ at $Z = 0, h(y)$ lead to
\begin{equation}
\label{U1}
v = \frac{1}{2 \mu} \partial_y p(\vr) (Z^2 - Zh) + s \Bigl(1 - \frac{Z}{h} \Bigr),
\end{equation}
Integrating over \(Z\), we find
\begin{equation*}
\int_0^h v\dd Z=-\frac{h^3}{12\mu}\partial_y p(\vr)+\frac{sh}2,
\end{equation*}
from which we deduce $\partial_y p(\vr) \in \L^2(\mathcal{T}^1)$.
In particular $p(\vr)$ is H\" older continuous and periodic in $y$. It follows that
\begin{equation*}
\sup_{y \in \mathcal{T}^1} \vr(y) < \bar{\vr}.
\end{equation*}
According to \eqref{L2}, we obtain the following ODE
\begin{equation}
\label{U2}
\partial_y \Bigl( \frac{h^3}{12 \mu} \vr \partial_y p(\vr) \Bigr)
= \partial_y \Bigl(\frac{\vr h s}{2}\Bigr) \quad\text{in}\quad\mathcal{D}'(\mathcal{T}^1)
\quad\text{and}\quad
 \int_{\mathcal{T}^1} h \vr
\ {\rm d}y = M.
\end{equation}
We show that problem \eqref{U2} admits at most one solution. First, we have
\begin{equation}
\label{U3}
\frac{h^3}{12 \mu} \vr \partial_y p(\vr) = \frac{\vr h s}2  + \lambda\quad\text{with}\quad \lambda \in \Er,
\end{equation}
which implies that
\begin{equation}
\label{U33}
\frac{1}{12 \mu} \vr \partial_y p(\vr) = \frac{\vr  s}{2 h^2}
+ \frac{\lambda}{h^3}\quad\text{with}\quad \lambda \in \Er.
\end{equation}
By using the following identity:
\begin{equation*}
\vr p'(\vr)\partial_y\vr=\partial_y f(\vr)\quad\text{with}\quad f(\vr)\eqldef\int_1^\vr s p'(s)\,\dd s,
\end{equation*}
we obtain
\begin{equation*}
\int_{\mathcal{T}^1}\vr \partial_y p(\vr)\,\dd y = 0.
\end{equation*}
Since we have
\begin{equation*}
\int_{\mathcal{T}^1}\frac{\vr}{h^2}\,\dd y>0\quad\text{and}\quad
\int_{\mathcal{T}^1}\frac1{h^3}\,\dd y>0,
\end{equation*}
we conclude from \eqref{U33} that \(\lambda<0\).
Now we prove that \(\vr>0\). We proceed by contradiction; we assume that
there exists \(y_0\in\mathcal{T}^1\) such that
\begin{equation}
\label{U34}
\vr(y_0)=0.
\end{equation}
We may deduce from \eqref{U3} and \(\partial_yp(\vr)\in\L^2(\mathcal{T}^1)\) that
\(\vr(y)>0\) a.e. \(y\in\mathcal{T}^1\). Then it follows from \eqref{U33} that
\begin{equation}
\label{U45}
\vr'=
\Bigl(
\frac{6\mu s}{h^2}+\frac{12\mu\lambda}{h^3\vr}
\Bigr)\frac1{p'(\vr)}\quad\text{a.e. on}\quad\mathcal{T}^1.
\end{equation}
We observe that we can take \(y_0\leq 1\). Since \(\vr\) is a continuous function,
\(\vr(y_0)=0\) and \(\lambda<0\), it comes by using \eqref{U45} that there exists
\(\eta>0\) small enough satisfying \(y_0+\eta\in\mathcal{T}^1\) such that
\(\vr'(y)<0\) a.e. \(y\in(y_0,y_0+\eta)\). But
\begin{equation*}
\vr(y_0+\eta)=\int_{y_0}^{y_0+\eta} \vr'(y)\,\dd y<0
\end{equation*}
gives the contradiction.
We conclude that \(\vr>0\).
Accordingly, equation \eqref{U3} becomes a non--degenerate first order non--linear ODE
with $\C^1$ non--linearities.
Next, we claim that for a given $\lambda<0$,
there is at most one periodic solution \(\vr\) of \eqref{U3} satisfying also
\begin{equation}
\label{U36}
\int_{\mathcal{T}^1} \vr h\,\dd y = M.
\end{equation}
Indeed, two possible such solutions do not intersect, therefore
there exists exactly one of them for which \eqref{U36} holds.
Let now $\vr_1$ and $\vr_2$ - two periodic solutions - of \eqref{U3} corresponding to two
different constants \(\lambda_1\neq \lambda_2\) and satisfying \eqref{U36}.
To fix the ideas, we suppose that \(\lambda_2>\lambda_1\). It follows from \(p'>0\) on
\((0,\bar\vr)\) and \eqref{U3} that
\begin{equation}
\label{U37}
\vr_1'(y)>\vr'_2(y)\text{ for any }y\in[0,1]\text{ such that }\vr_1(y)=\vr_2(y)>0.
\end{equation}
Since \eqref{U36} holds, there exists \(\alpha\in (0,1)\) such that \(\vr_1(\alpha)=\vr_2(\alpha)\).
We conclude from \eqref{U37} that we have
\begin{equation*}
\forall y\in [0,\alpha): \vr_1(y)>\vr_2(y)\quad\text{and}\quad
\forall y\in (\alpha,1]: \vr_1(y)<\vr_2(y),
\end{equation*}
which contradicts the periodicity of \(\vr_1\) and \(\vr_2\).

\appendix

\section{Auxiliary results, Bogovskii's operator, extension lemma}
\label{a}

\subsection{Bogovskii's operator}

 The goal is to find a solution branch of the obviously undetermined problem,
\begin{equation}
 \label{a1}
\Div \vc{B} = f - \frac{1}{|\Omega|} \intO{ f }\quad\text{in}\quad\Omega
\quad\text{and}\quad\vc{B}=0\quad\text{on}\quad{\partial \Omega},
\end{equation}
where $\Omega$ is a bounded Lipschitz domain in $\Er^N$.

In this paper, we consider the operator $\vc{B} = \vc{B}[f]$
proposed by Bogovskii \cite{BOG}.
The following properties of $\vc{B}$ were proved
in the monograph by Galdi \cite{GAL}:

\begin{itemize}
\item
The operator $f \mapsto \vc{B}[f]$ is linear, a priori defined from
$\C^{\infty}(\bar{\Omega})\) to \( \C^{\infty}_c(\Omega)$.
\item $\vc{B}$ can be extended to functions $f \in \L^p(\Omega)$,
\begin{equation}
\label{a2}
\| \vc{B}[f] \|_{\W^{1,p}_0 (\Omega)} \leq c(p) \| f \|_{\L^p(\Omega)}\quad\text{for any}\quad
1< p < \infty.
\end{equation}

\item If, in addition, $f = \Div (\vc{g})$, where $g \in \L^q$ such that
$\vc{g} \cdot \vc{n}= 0$ on \(\partial\Omega\) holds, then
\begin{equation}
\label{a3}
\| \vc{B}[\Div (\vc{g})] \|_{\L^q(\Omega)} \leq c(q) \| \vc{g} \|_{\L^q(\Omega)}.
\end{equation}

\end{itemize}

\subsection{Extension operator}

Next, we claim the following extension lemma.
\begin{lemma}
\label{aL1}
Let $\Omega \subset \Er^N$, $N=2,3$, be a bounded Lipschitz domain.
Let $\bar{\vu}\in\W^{1,p}(\Omega)$ be given such that $p\in]0,+\infty[$ and
$\bar{\vu} \cdot \vc{n}= 0$ on \(\partial \Omega\). Let \(q\) be given such that
\(1< q < \frac{Np}{N - p}\) if \(p < N\)
and \(1 < q\) arbitrary finite otherwise. Then for any $\delta > 0$, there exists
$\bar{\vu}_\delta \in \W^{1,p}(\Omega)$ with the following properties:
\begin{itemize}
\item \(\bar{\vu}_\delta = \bar{\vu}\) on \(\partial \Omega\) in the sense of traces
and \(\Div \bar{\vu}_\delta = 0\) in \(\Omega\),

\item \(\| \bar{\vu}_\delta \|_{\L^q(\Omega)} < \delta\)
and \(\| \bar{\vu}_\delta \|_{\W^{1,p}(\Omega)}
\leq c(\delta, p,q) \| \bar{\vu} \|_{\W^{1,p}(\Omega)}\),

\end{itemize}

\end{lemma}

\begin{proof}
As $\W^{1,p}(\Omega) \hookrightarrow \L^q(\Omega)$,
it is easy to construct an extension $\tilde{\vu}_\delta$ such that
\begin{equation*}
\tilde{\vu}_\delta = \bar{\vu} \quad\text{on}\quad\partial \Omega,\quad
\| \tilde{\vu}_\delta \|_{\L^q(\Omega)}
< \delta \quad\text{and}\quad
 \| \tilde{\vu}_\delta \|_{\W^{1,p} (\Omega)} \leq c(\delta, p,q) \| \bar{\vu} \|_{\W^{1,p}(\Omega)}.
\end{equation*}
Now, consider the Bogovkii's operator solving
\begin{equation*}
\Div (\vc{B}) = \Div (\tilde{\vu}_\delta)\quad\text{and}\quad\vc{B}= 0\quad\text{on}\quad\partial\Omega.
\end{equation*}
It follows from the basic properties of the operator $\vc{B}$ stated in \eqref{a2} and \eqref{a3}
that
\begin{equation*}
\bar{\vu}_\delta = \tilde{\vu}_\delta - \vc{B}
\end{equation*}
satisfies
\begin{subequations}
\begin{align}
&\bar{\vu}_\delta = \bar{\vu}\quad\text{on}\quad\partial \Omega
\quad\text{and}\quad\Div (\bar{\vu}_\delta) = 0 \quad\text{in}\quad \Omega,
\\&
\| \bar{\vu} \|_{\W^{1,p}(\Omega)}\leq \| \tilde{\vu}_{\delta} \|_{\W^{1,p}(\Omega)}
+ \| \vc{B} \|_{\W^{1,p}(\Omega)}\\&\notag \leq c_1(\delta,p,q) \bigl(\| \tilde{\vu}_{\delta} \|_{\W^{1,p}
(\Omega)} +
\| \Div \tilde{\vu}_\delta \|_{\L^p(\Omega)} \bigr)
\leq c_2(\delta,p,q) \| \tilde{\vu}_{\delta} \|_{\W^{1,p}(\Omega)},
\\&
\| \bar{\vu}_\delta \|_{\L^q(\Omega)} \leq \| \tilde{\vu}_\delta \|_{\L^q(\Omega)} +
\| \vc{B} \|_{\L^q(\Omega)} \leq \delta + c(q) \| \tilde{\vu}_\delta \|_{\L^q(\Omega)}
\leq \delta (1 + c(q)).
\end{align}
\end{subequations}\\
\end{proof}

\subsection{Anisotropic interpolation}

Consider the domain
\begin{equation*}
\Omega\eqldef\bigl{\{} x = (x_h,z):\, x_h \in \mathcal{T}^2 ,\ 0 < z < h(x_h) \bigr{\}},
\end{equation*}
where $h$ is a Lipshitz function, see \cite{ABC94IETF}.
Then we have
\begin{equation}
\label{a4}
\| \vc{v} \|_{\L^4(\Omega)} \leq \bigl(\| \nabla_h \vc{v} \|_{\L^2(\Omega)}
+ \| \vc{v} \|_{\L^2(\Omega)} \bigr)^{\frac{1}{2}} \bigl(\| \partial_Z \vc{v} \|_{\L^2(\Omega)}
+ \| \vc{v} \|_{\L^2(\Omega)}\bigr)^{\frac{1}{2}}
\end{equation}
for any $\vc{v} \in\W^{1,2}(\Omega)$ such that
$\vc{v}(\cdot,z = h(x))= 0$ and $\vc{v} \cdot \vc{n}= 0$ on \(\partial \Omega\) hold.

\section{Existence and uniqueness results to the approximate problem}
\label{AppendixA}

We prove here the following lemma
\begin{lemma}
\label{exists_uniqueness}
Assume that \(\mathbf{u}\in\W^{1,2}(\Omega)\) such that \(\mathbf{u}\cdot \mathbf{n}=0\) on
\(\partial\Omega\) holds and \(g\in\L^2(\Omega)\).
Then there exists a unique solution \(\vr\in\W^{2,2}(\Omega)\)
such that \(\vr\geq 0\) to problem
\begin{subequations}
\label{appendixA1}
\begin{align}
&\delta\vr-\delta\Delta
\vr+\dive(T(\vr)\mathbf{u})=g\quad\text{in}\quad\Omega,\\
&\nabla\vr\cdot\mathbf{n}=0\quad\text{on}\quad\partial\Omega.
\end{align}
\end{subequations}
Furthermore, we have the following comparison principle:
for any \(g_1\in\L^2(\Omega)\) and \(g_2\in\L^2(\Omega)\) with \(g_1\geq g_2\),
we have \(\vr_1\geq \vr_2\) where \(\vr_1\) and \(\vr_2\) are the solutions to problem
\eqref{appendixA1} corresponding to \(g_1\) and \(g_2\), respectively.
\end{lemma}

\begin{proof}
The existence results comes from the Schauder's fixed point theorem. To this aim,
let \(\mathcal{S}\) be the following mapping:
\begin{equation*}
\begin{aligned}
\mathcal{S}:\,\W^{1,2}(\Omega)&\rightarrow \W^{1,2}(\Omega)\\
\tilde{\vr}&\mapsto \vr=\mathcal{S}(\tilde{\vr})
\end{aligned}
\end{equation*}
where \(\vr\) is assumed to be a solution to the problem
\begin{equation*}
\begin{aligned}
&\delta\vr-\delta\Delta
\vr+\dive(T(\tilde{\vr})\mathbf{u})=g\quad\text{in}\quad\Omega,\\
&\nabla\vr\cdot\mathbf{n}=0\quad\text{on}\quad\partial\Omega.
\end{aligned}
\end{equation*}
Since \(\vu\cdot\mathbf{n}=0\) on \(\partial\Omega\), the variational formulation for \(\vr\)
is given by
\begin{equation*}
\forall \psi\in\W^{1,2}(\Omega):\,
\delta\int_{\Omega}\vr\psi\,\dd x+\delta\int_{\Omega}\nabla\vr\nabla\psi\,\dd x-
\int_{\Omega}(T(\tilde{\vr})\vu)\cdot\nabla\psi\,\dd x=\int_{\Omega}g\psi\,\dd x,
\end{equation*}
with \(\vr\in\W^{1,2}(\Omega)\). Observe that
\begin{equation*}
\norm[\L^2(\Omega)]{T(\tilde{\vr})\mathbf{u}}\leq\bar{\vr}\norm[\L^2(\Omega)]{\vu},
\end{equation*}
which implies by using Lax--Milgram theorem that there exists a constant \(C\geq 0\)
independent of \(\tilde{\vr}\) such that
\begin{equation*}
\norm[\W^{1,2}(\Omega)]{\mathcal{S}(\tilde{\vr})}\leq C.
\end{equation*}
Then \(\mathcal{S}(\bar{B}_{\W^{1,2}(\Omega)}(0,C))\subset
\bar{B}_{\W^{1,2}(\Omega)}(0,C)\)
where \(\bar{B}_{\W^{1,2}(\Omega)}(0,C)\) denotes a closed ball in \(\W^{1,2}(\Omega)\)
with radius \(C\). Since \(T\in\W^{1,\infty}(\Er)\) and \(\vu\in\L^6(\Omega)\),
we may deduce that \(\dive(T(\tilde{\vr})\vu)\) is bounded in \(\L^{\frac32}(\Omega)\)
by a constant independently of \(\tilde{\vr}\). It follows by using a regularity argument that
\(\vr\) is bounded in \(\W^{2,\frac32}(\Omega)\) and it comes from the Sobolev compactness that
\(\mathcal{S}(\bar{B}_{\W^{1,2}(\Omega)}(0,C))\) is relatively compact in \(\W^{1,2}(\Omega)\).
On the other hand, assume that \(\tilde{\vr}_n\) converges strongly to \(\tilde{\vr}\) in
\(\W^{1,2}(\Omega)\) as \(n\) tends to \(+\infty\). Since \(T\) is Lipschitz continuous and
\(\W^{1,2}(\Omega)\hookrightarrow\L^6(\Omega)\), it is clear that
\begin{equation*}
\norm[\L^6(\Omega)]{T(\tilde{\vr}_n)-T(\tilde{\vr})}\rightarrow 0.
\end{equation*}
Since \(\vu\in\L^3(\Omega)\), we may infer that
\(\vr_n\) converges to \(\vr\) in
\(\W^{1,2}(\Omega)\) as \(n\) tends to \(+\infty\) and then \(\mathcal{S}\) is continuous.
Hence we may conclude by using Schauder's fixed point theorem
(see \cite{Evans90PDE}) that \eqref{appendixA1} possesses a solution.

We establish now the comparison principle which gives also the uniqueness result
associated to our problem.
We introduce the following notations:
\(\vr\eqldef \vr_1-\vr_2\) and \(g\eqldef g_1-g_2\), then we have
\begin{subequations}
\label{appendixA2}
\begin{align}
\label{appendixA21}
&\delta\vr-\delta\Delta
\vr+\dive((T(\vr_1)-T(\vr_2))\mathbf{u})=g\quad\text{in}\quad\Omega,\\
\label{appendixA22}
&\nabla\vr\cdot\mathbf{n}=0\quad\text{on}\quad\partial\Omega,
\end{align}
\end{subequations}
where \(g\) is a positive function.  Let us define \(\varphi\in\W^{1,\infty}(\Er)\) such that
\(\varphi(0)=0\) and \(\varphi'(\zeta)>0\) for all \(\zeta\in[0,+\infty)\). Furthermore, \((\cdot)^{+}\)
and \((\cdot)^{-}\) denote the positive and negative part of \((\cdot)\), respectively.
Multiplying \eqref{appendixA21} by \(\varphi(\vr^-)\) and integrating over \(\Omega\) the resulting
identity, we find
\begin{equation*}
\delta\int_{\Omega}\vr^-\varphi(\vr^-)\,\dd x
+\delta\int_{\Omega}\varphi'(\vr^-)\abs{\nabla\vr^-}^2\,\dd x
\leq \int_{\Omega}\abs{\mathbf{u}}\vr^-\varphi'(\vr^-)\abs{\nabla\vr^-}\,\dd x,
\end{equation*}
which leads to
\begin{equation*}
\delta\int_{\Omega}\vr^-\varphi(\vr^-)\,\dd x
\leq\frac1{2\delta}\int_{\Omega}\abs{\vu}^2(\vr^-)^2\varphi'(\vr^-)\,\dd x.
\end{equation*}
Choosing \(\varphi(\zeta)=\frac{\zeta}{\zeta+\eta}\), we have
\begin{equation*}
\int_{\Omega}\frac{(\vr^-)^2}{\eta+\vr^-}\,\dd x\leq \frac{\eta}{2\delta^2}\int_{\Omega}
\abs{\mathbf{u}}^2\,\dd x.
\end{equation*}
We may deduce that \(\int_{\Omega}\vr^-\,\dd x=0\)
which ends the proof of the comparison principle.\\
\end{proof}

\renewcommand{\arraystretch}{0.91}{\small{
\paragraph*{Acknowledgments}
The research of E.~F. leading to these results has received funding from the
European Research Council under the European Union's Seventh
Framework Programme (FP7/2007-2013)/ ERC Grant Agreement
320078. The Institute of Mathematics of the Academy of Sciences of
the Czech Republic is supported by RVO:67985840.
This research was performed during the stay of E.~F. as invited professor at the INSA--Lyon.

\end{document}